\documentclass{amsart}
\usepackage{amsthm,amsmath,amsfonts,amssymb,amscd,mathrsfs,latexsym,mathtools}
\usepackage{enumerate,graphicx,cases,extarrows,indentfirst,tabularx}
\usepackage{slashed}

\usepackage{xcolor}
\usepackage[all,color,2cell]{xy}
\usepackage{hyperref} 

\usepackage{mathpazo}
\newcommand{\C}{\mathbb{C}}
\newcommand{\N}{\mathbb{N}}
\newcommand{\R}{\mathbb{R}}

\newcommand{\Z}{\mathbb{Z}}

\newcommand{\Spec}{\text{Spec}}

\newcommand{\abracket}[1]{\left\langle#1\right\rangle}

\newcommand{\pa}{\partial}

\newcommand{\into}{\hookrightarrow}

\newcommand{\iso}{\cong}

\newcommand{\ST}{\mathscr{S}}
\newcommand{\XT}{\mathscr{X}}
\newcommand{\RT}{\mathscr{R}}
\newcommand{\fvol}{\mathbb{1}}

\newtheorem{theorem}{Theorem}[section]
\newtheorem{proposition}[theorem]{Proposition}
\newtheorem{lemma}[theorem]{Lemma}
\newtheorem{remark}[theorem]{Remark}
\newtheorem{definition}[theorem]{Definition}

\newtheorem{corollary}[theorem]{Corollary}

\newcommand{\A}{\mathcal{A}}
\newcommand{\cO}{\mathcal{O}}

\newcommand{\pat}{\partial}

\newcommand{\td}{\text{d}}

\DeclareMathOperator{\PV}{PV}

\DeclareMathOperator{\KK}{K}
\DeclareMathOperator{\Res}{Res}
\DeclareMathOperator{\Tr}{Tr}

\def\Xint#1{\mathchoice
{\XXint\displaystyle\textstyle{#1}}%
{\XXint\textstyle\scriptstyle{#1}}%
{\XXint\scriptstyle\scriptscriptstyle{#1}}%
{\XXint\scriptscriptstyle\scriptscriptstyle{#1}}%
\!\int}
\def\XXint#1#2#3{{\setbox0=\hbox{$#1{#2#3}{\int}$}
\vcenter{\hbox{$#2#3$}}\kern-.5\wd0}}

\def\dashint{\Xint-}

\begin{document}

\title[Primitive forms from log LG toric mirrors]{A perturbative construction of primitive forms from log Landau-Ginzburg mirrors of toric manifolds}

\author[Chan]{Kwokwai Chan}
\address{Department of Mathematics\\ The Chinese University of Hong Kong\\ Shatin\\ Hong Kong}
\email{kwchan@math.cuhk.edu.hk}
 
\author[Ma]{Ziming Nikolas Ma}
\address{Department of Mathematics\\ Southern University of Science and Technology\\ Shenzhen\\ China}
\email{mazm@sustech.edu.cn}

\author[Wen]{Hao Wen}
\address{School of Mathematical Sciences and the Key Laboratory of Pure Mathematics
	and Combinatorics\\ Nankai University\\ Tianjin\\ China}
\email{wenhao@nankai.edu.cn}

\begin{abstract}
    We introduce the notion of a logarithmic Landau-Ginzburg (log LG) model, which is essentially given by equipping the central degenerate fiber of the family of Landau-Ginzburg (LG) models mirror to a projective toric manifold with a natural log structure. 
    We show that the state space of the mirror log LG model is naturally isomorphic to that of the original toric manifold. 
    Following \cite{LLS, LLSS}, we give a perturbative construction of primitive forms by studying the deformation theory of such a log LG model, which involves both smoothing of the central degenerate fiber and unfolding of the superpotential.
    This yields a logarithmic Frobenius manifold structure on the base space of the universal unfolding.
    The primitive forms and flat coordinates we obtained are computable and closely related to the bulk-deformed Lagrangian Floer superpotential of a projective toric manifold, at least in the semi-Fano case.
\end{abstract}

\maketitle

\tableofcontents

\section{Introduction}

\subsection*{Background}
The long history of mirror symmetry for toric manifolds can be traced back to the works of Batyrev \cite{Batyrev93}, Givental \cite{Givental98}, Lian-Liu-Yau \cite{LLY-III}, Kontsevich \cite{Kontsevich-ENS98}, Hori-Vafa \cite{Hori-Vafa00}, and many others.
Given an $n$-dimensional projective toric manifold $Y$ defined by a complete smooth fan $\Sigma \subset \R^n$, its mirror is generally agreed to be given by a {\em Landau-Ginzburg (abbrev. LG) model} $(X_1, f)$, where $X_1$ is a smooth affine variety isomorphic to the algebraic torus $(\C^*)^n$ and the so-called {\em superpotential} $f: X_1 \to \C$ is a Laurent polynomial whose Newton polytope is the fan polytope of $\Sigma$.
As an example, the LG mirror of $Y = \mathbb{P}^2$ is given by $ X_1 = \{(z_1, z_2, z_3) \in \C^3: z_1 z_2 z_3 = 1 \}\cong (\C^*)^2$ together with the restriction $f$ of the function $z_1 + z_2 + z_3$ to $X_1 \subset \C^3$.

In genus zero, mirror symmetry can be understood as an isomorphism between \emph{Frobenius manifolds}.
The classical work of K. Saito \cite{Saito83} first introduced \emph{primitive forms} and used them to construct flat structures associated to universal unfoldings of isolated hypersurface singularities.
Dubrovin \cite{Dubrovin93} generalized and unified this with the WDVV equations in genus zero Gromov-Witten (GW) theory by introducing the notion of Frobenius manifolds.
We call the Frobenius manifold coming from genus zero GW theory (or big quantum cohomology) the \emph{A-model Frobenius manifold}.
Douai-Sabbah \cite{douai2003gauss, douai2004gauss} extended Saito's work to a broad class of examples, enabling the construction of the \emph{B-model Frobenius manifold} from the LG mirror $(X_1, f)$.
Genus zero toric mirror symmetry can then be phrased as an isomorphism (possibly via a nontrivial mirror map) between this B-model Frobenius manifold and the A-model Frobenius manifold associated to $Y$.
In \cite{barannikov2000semi}, Barannikov established such an isomorphism in the case of projective spaces (using his own construction of the B-model Frobenius manifold).\footnote{Barannikov's construction was based on an earlier construction by Barannikov-Kontsevich \cite{BK98} which produced a Frobenius manifold structure on the extended moduli space of a Calabi-Yau manifold via polyvector fields; see also Li-Wen \cite{LW22} where they gave a construction of Frobenius manifolds via the Barannikov-Kontsevich approach unifying the Landau-Ginzburg and Calabi-Yau geometry.}

Subsequently, Iritani \cite{Iritani08, Iritani17} and Coates-Corti-Iritani-Tseng \cite{CCIT20} investigated mirror symmetry for toric manifolds (and more generally, toric Deligne-Mumford stacks) in terms of quantum $\mathcal{D}$-modules. They also constructed primitive forms and described the B-model Frobenius manifold structure. On the other hand, Reichelt \cite{R} introduced the notion of a \emph{logarithmic Frobenius manifold}. Reichelt-Sevenheck \cite{RS} showed the existence of a primitive form, yielding a logarithmic Frobenius structure associated to the LG model $(X_1, f)$. In fact, all these constructions of primitive forms and (log) Frobenius structures made use of the \emph{mirror family} of LG models $(\mathscr{X}, F)$ in which $(X_1, f)$ is a smooth fiber. For example, the LG mirror family for $Y = \mathbb{P}^2$ is given by $\mathscr{X} = \{(z_1, z_2, z_3; \tau) \in \C^3 \times \C: z_1 z_2 z_3 = \tau \} \subset \C^3 \times \C$ together with the superpotential given by restricting $F = z_1 + z_2 + z_3$ to each fiber $X_\tau$.

\subsection*{Main results}
In this paper, we introduced the notion of a \emph{logarithmic Landau-Ginzburg model} (abbrev. \emph{log LG model}), which is essentially given by equipping the central (singular) fiber $X = X_0$ of the mirror family $\mathscr{X}$ with a natural \emph{log structure} (see Definition \ref{log LG model}). 
For instance, the log LG model mirror to $Y = \mathbb{P}^2$ is given by equipping the singular surface $X = X_0 = \{(z_1, z_2, z_3) \in \C^3: z_1 z_2 z_3 = 0 \}$ (a union of 3 copies of $\mathbb{C}^2$) with a log structure and the superpotential is the restriction $f$ of the function $z_1 + z_2 + z_3$ to $X_0 \subset \C^3$.

Our main result is a perturbative construction of primitive forms from the log LG model mirror to a projective toric manifold, and therefore a construction of a logarithmic Frobenius manifold structure on the base of its universal unfolding (see Section \ref{sec:deformation} and the proof of Theorem \ref{main theorem} for details):
\begin{theorem}\label{main theorem intro}
Let $(X^\dag, \phi, f)$ be the logarithmic Landau-Ginzburg model mirror to a projective toric manifold $Y$. Then there is a perturbative construction of primitive forms which yields a logarithmic Frobenius manifold structure on the base space $\ST$ of the universal unfolding of $(X^\dag, \phi, f)$.
\end{theorem}
The construction of primitive forms and Frobenius manifold structures is along the lines of the perturbative, explicit approach using polyvector fields pioneered by Li-Li-Saito \cite{LLS} (see also \cite{MSaito89, MSaito18}), which is a complex differential-geometric framework inspired by the BCOV gauge theory on Calabi-Yau manifolds \cite{BCOV}. Indeed, our results can be viewed as the second application of the general construction and explicit perturbative formula in \cite{LLS}, after that of exceptional unimodular singularities in \cite{LLSS}.
More precisely, our primitive forms and Frobenius manifold structures can all be determined perturbatively and explicitly from the central degenerate fiber $X = X_0$ of the mirror LG family (or the log LG mirror). This is in sharp contrast with previous constructions, which rely on either the entire or punctured mirror LG family. 

As $X = X_0$ is non-compact, we start the construction by following \cite{LLS} to transform the complex $\PV_{\log}(X)$ of log polyvector fields to that $\PV_{\log, c}(X)$ with compact support via a homotopy (Section \ref{homotopy formula}). This allows us to define the \emph{higher residue pairing} $K_f$ on the (formally completed) \emph{Brieskorn lattice} $\mathcal{H}_f := H(\PV_{\log}(X)[[u]], \bar\pat + df \wedge + u\pat)$ using integration (Section \ref{integration}).
Now the main advantage and novelty in choosing the central degenerate fiber $X = X_0$ as our reference fiber is that we can find particularly simple and explicit \emph{good bases} $\{\varphi_i\}$ (Theorem \ref{good basis}) via the state space isomorphism (Theorem \ref{state space isomorphism}). In general, the existence of good basis is a highly nontrivial problem.
This greatly simplifies the construction of the logarithmic Frobenius structure.

Next we study an unfolding of the log LG model $(X,f)$, which consists of two directions -- one corresponds to the \emph{smoothing} of the singular Calabi-Yau variety $X$ while the other corresponds to \emph{unfolding} of the superpotential $f$. 
We take a \emph{universal unfolding} $(\mathscr{X}, F)$ of $(X, f)$ (Section \ref{unfolding}). Applying a slight generalization of the theory of regularized integrals due to Li-Zhou \cite{LZ} (which is needed because we will encounter integrands which have arbitrary orders of poles; see Appendix \ref{Section of regularized integral}), we can extend the higher residue pairing $K_f$ to $K_F$ and prove that it is compatible with the extended Gauss-Manin connection $\nabla$ (Section \ref{subsec:GM_connection_and_higher_residue}). This gives the main result in Section \ref{subsec:GM_connection_and_higher_residue}, namely, the triple $(\mathcal{H}_F, \nabla, \KK_F)$ forms a \emph{log semi-infinite variation of Hodge structure} (abbrev. \emph{$\frac{\infty}{2}$-LVHS}); see Theorem \ref{properties of HRP}.

At this point, the remaining ingredients in constructing a log Frobenius manifold are a \emph{good opposite filtration} and a \emph{primitive form}. The existence of the former is proved using the \emph{Deligne extension} of a good basis $\{\varphi_i\}$ (Lemma \ref{HRP on extension of good basis} and Proposition \ref{good opposition}), while the latter is constructed using the standard \emph{Birkhoff factorization} procedure (Equation \eqref{spliting}, Lemma \ref{lemma_primitive_form}, Proposition \ref{prop:primitive_form}). From these we obtain a good opposite filtration and a primitive form $\zeta$, hence yielding a logarithmic Frobenius manifold structure on the base of the universal unfolding and completing the proof of Theorem \ref{main theorem intro}; see Section \ref{log Frobenius} for more details.





\subsection*{Relation with the toric A-side}
As our construction is within the framework of Li-Li-Saito \cite{LLS}, the primitive forms and \emph{flat coordinates} (or \emph{semi-infinite period maps}) can all be computed perturbatively and explicitly. 
If we write the primitive form as $\zeta = e^{\tilde{F}/u}$, then the \emph{corrected} superpotential $f + \tilde{F}$, expressed in terms of the flat coordinates (and with the descendant variable set to be zero), should be closely related to the \emph{bulk-deformed Lagrangian Floer superpotential} \cite{FOOO-toricI, FOOO-toricII} of the projective toric manifold $Y$. 

There are very few explicitly known calculations of the bulk-deformed Lagrangian Floer superpotential of a toric manifold. Gross \cite{Gross10, gross2011tropical} constructed an explicit perturbation of the Hori-Vafa potential $f = z_1 + z_2 + z_3$ mirror to $Y = \mathbb{P}^2$ using counts of Maslov index 2 {\em tropical disks}. Recently, Hong-Lin-Zhao \cite{hong2018bulk} generalized this and provided an inductive algorithm to compute the bulk-deformed Lagrangian Floer superpotentials for all Fano toric surfaces using wall-crossing techniques and a tropical-holomorphic correspondence.

For non-bulk-deformed Lagrangian Floer superpotentials (meaning that all the `big quantum variables' are set to be zero; see Section \ref{semi-Fano}), the first calculation was done by Cho-Oh \cite{Cho-Oh06} in the Fano toric case.
Beyond the Fano setting, Auroux \cite{Auroux09} and Fukaya-Oh-Ohta-Ono \cite{FOOO12} first computed the Lagrangian Floer superpotential for the Hirzebruch surface $\mathbb{F}_2$ (Auroux also worked out the $\mathbb{F}_3$ example). Chan-Lau extended this to all semi-Fano toric surfaces in \cite{Chan-Lau10}. More generally, Chan-Lau-Leung-Tseng \cite{CLLT12} and Gonzale\'z-Iritani \cite{Gonzalez-Iritani17} proved that the non-bulk-deformed Lagrangian Floer superpotential of a semi-Fano toric manifold only differs from its Hori-Vafa superpotential by the toric mirror map (a coordinate change), thereby giving an explicit calculation of the non-bulk-deformed superpotential in all such cases. 

In the semi-Fano case, in fact we can see by a weight degree argument that the primitive form is simply given by $\zeta = 1$ when all `big quantum variables' are set to be zero (see Section \ref{semi-Fano}). In this case, we expect that our period map coincides with the toric mirror map and the non-bulk-deformed Lagrangian Floer superpotential of a semi-Fano toric manifold $Y$ is simply given by the Hori-Vafa potential $f$ expressed in terms of the flat coordinates (and restricted to the `small quantum variables'). We verify that this is indeed the case for the Hirzebruch surface $\mathbb{F}_2$ in Section \ref{F2}.

\begin{remark}
    The explicit computations in Section \ref{F2} are all done by hand. In general, the semi-infinite period map and primitive form (and hence the corrected superpotential) of any log LG toric mirror can all be computed up to any desired order using a computer program. 
\end{remark}

\subsection*{Organization of the paper}
The rest of this paper is organized as follows.
In Section \ref{constructions}, we define the logarithmic Landau-Ginzburg model (abbrev. log LG model) $(X^\dag, \phi,  f)$ mirror to a projective toric manifold $Y$.
In Section \ref{state space isom}, we prove the \emph{state space isomorphism} and deduce the \emph{Hodge-to-de Rham degeneration} (Theorem \ref{Hodge-to-de Rham}).
After transforming smooth differential forms/polyvector fields into those with compact support by a homotopy in Section \ref{homotopy formula}, following \cite{LLS}, we define the \emph{higher residue pairing} using integration in Section \ref{integration}.
In Section \ref{grading}, we introduce grading structures, and prove the crucial \emph{Moving Lemma} (Lemma \ref{algebraic moving lemma}) which is the key for constructing a natural, explicit \emph{good basis} (Theorem \ref{good basis} and Proposition \ref{nondegeneracy}).

In Section \ref{unfolding}, we study the unfolding/deformation of our log LG model and deduce the \emph{freeness of the Hodge bundle} (Theorem \ref{freeness of Hodge bundle}).
In Section \ref{subsec:GM_connection_and_higher_residue}, we prove the compatibility between the \emph{Gauss-Manin connection} and the higher residue pairing, which leads to the construction of a \emph{log semi-infinite variation of Hodge structure} (abbrev. \emph{$\frac{\infty}{2}$-LVHS}) $(\mathcal{H}_F, \nabla, \KK_F)$ in Theorem \ref{properties of HRP}.
In Section \ref{log Frobenius}, we construct a \emph{good opposite filtration} (Proposition \ref{good opposition}) and a \emph{primitive form} (Proposition \ref{prop:primitive_form}), thereby producing the \emph{logarithmic Frobenius manifold} structure.
In Section \ref{semi-Fano}, we further study the case when $Y$ is \emph{semi-Fano} (i.e., the anticanonical bundle $K_Y^{-1}$ is nef). We show that the period map restricted to the small quantum variables (i.e., those which correspond to $H^2(Y)$) is of the shape of the toric mirror map (Theorem \ref{small period map}), as expected. Finally, explicit calculations for the Hirzebruch surface $\mathbb{F}_2$ are provided in Section \ref{F2}.


\subsection*{Acknowledgement}

We would like to thank Hiroshi Iritani, Changzheng Li, Si Li and Junwu Tu for very helpful discussions.
K. Chan was supported by grants from the Hong Kong Research Grants Council (Project No. CUHK14301420, CUHK14301621 \& CUHK14305322).
Z. N. Ma was supported by National Natural Science Foundation of China (Grant No. K23281001, K23281103).
H. Wen was supported by the Young Scientists Fund of the National Natural Science Foundation of China (Grant No. 12201314).

\subsection*{Notation summary}

In this subsection, we explain some notation that will be frequently used.
When $P$ is a monoid, $\C[P]$ denotes the monoid ring of $P$ with coefficients in $\C$. We will work with analytic spaces and write $\text{Spec}(\C[P])$ as the corresponding toric analytic variety. When $X$ is a log analytic space, we will write $X^\dag$ to emphasize its log structure.
When $\phi: X^\dag \to Y^\dag$ is a map of log spaces, $\Omega^1_{X^\dag / Y^\dag}$ and $\Theta^1_{X^\dag / Y^\dag}$ denote the sheaf of relative log differentials and sheaf of relative log derivations on $X$ respectively.
Let $u$ be a formal variable and $A$ be a $\C$-vector space. Then $A[[u]]$ and $A((u))$ denote the space of $A$-valued formal power series and formal Laurent power series in $u$ respectively.

\section{The log LG model mirror to a projective toric manifold}\label{section: log LG models}

\subsection{Constructions} \label{constructions}

In this paper, we assume the reader has some familiarity with log geometry; for references, see \cite{Ka1, Ka2}.
We first describe the local model used in the mirror symmetry program of Gross and Siebert, see \cite{Gross-Siebert-logI}.

Let $N \cong \Z^n$ be a lattice of rank $n$ and $N_\R := N \otimes _\Z \R$ the real vector space it spans.
Let $\Sigma \subset N_\R$ be a smooth complete fan which defines an $n$-dimensional smooth projective toric manifold $Y$.
Since $Y$ is projective, there is a strictly convex piecewise linear function $\psi$ on $|\Sigma|$ with only integral slopes.
The subset
\begin{equation*}
    P:=\{(n,r) | n \in |\Sigma|, r\geq \psi(n)\} \subset N_{\R}\oplus \R,
\end{equation*}
is a strictly convex rational polyhedral cone and hence defines an affine toric variety
\begin{equation*}
    \mathscr{X} := \text{Spec}(\C[P \cap (N\oplus \Z)]).
\end{equation*}

The natural inclusion $\N \hookrightarrow P \cap (N\oplus \Z)$ given by $q \to (0,q)$ is an inclusion of monoids, which induces a map $\mathscr{X} \to \C = \text{Spec}(\C[\N])$.
Let $X$ be the fiber over $0 \in \C$. Equip $\mathscr{X}$ with the log structure induced by the inclusion $X \into \mathscr{X}$ and equip $\C$ with the log structure induced by the inclusion $0\into \C$. Then $\Phi:\mathscr{X}^\dag \to \C^\dag$ is log smooth. By pulling back the log structures on $\mathscr{X}^\dag$ and $\C^\dag$ respectively, we get a log smooth map
\begin{equation*}
    \phi: X^{\dag} \to 0^\dag.
\end{equation*}

The space $X$ can be described more explicitly.
Let $\Sigma(1) := \{\rho_1,\cdots,\rho_d\} \subset N$ be primitive generators of the $1$-dimensional cones of $\Sigma$.
Recall that a subset $\mathcal{P}$ of $\Sigma(1)$ is a \emph{primitive collection} if $\mathcal{P}$ is not the set of generators of a cone of $\Sigma$, while for $0\leq k <p$, any $k$ elements of $\mathcal{P}$ generate a $k$-dimensional cone in $\Sigma$.
Each primitive collection $\mathcal{P} := \{\rho_{i_1},\cdots,\rho_{i_p}\}$ defines an equation $E_{\mathcal{P}}$ as follows
\begin{equation*}
    z_{i_1} \cdots z_{i_p} = 0.
\end{equation*}
When $\mathcal{P}$ runs over all primitive collections in $\Sigma(1)$, the equations $\{E_{\mathcal{P}}\}_{\mathcal{P}}$ define an affine space in $\C^d$, and this is exactly the central fiber $X$.
In other words, the ideal in $\C[z] := \C[z_1,\cdots,z_d]$ defining $X$ is the \emph{Stenley-Reisner ideal} $SR(\Sigma)$ of $\Sigma$.
Now $X = \cup_i X_i$ is the union of $n$-dimensional coordinate hyperplanes in $\C^d$, each corresponding to an $n$-dimensional cone of $\Sigma$. For convenience, we denote by $X_1$ the component defined by $z_l=0$ for $n+1 \leq l \leq d$.

Since $\phi: X^{\dag} \to 0^\dag$ is log smooth, the sheaf $\Omega^\bullet_{X^\dag / 0^\dag}$ of (relative) log differentials and 
the sheaf $\Theta^\bullet_{X^\dag / 0^\dag}$ of (relative) log derivations are both locally free on $X$.
Let us describe these sheaves explicitly.
Since $\Sigma$ is smooth, we can assume, without loss of generality, that $\rho_1=e_1,\cdots,\rho_n=e_n$ is the standard basis of $N$. Assume that when $n+1 \leq l \leq d$,
\begin{equation*}
    \rho_l = \sum_{i=1}^n a_{il} \rho_i.
\end{equation*}
By construction, we have
\begin{equation*}
    z_l = \prod_{i=1}^n z_i^{a_{il}} \cdot q^m
\end{equation*}
for some $m \in \Z_+$. By the definition of $\Omega^1_{X^\dag / 0^\dag}$, we have the relations
\begin{equation} \label{relation of log differential}
    \frac{\td z_l}{z_l} = \sum_{i=1}^n a_{il}\cdot \frac{\td z_i}{z_i}, \quad \forall n+1 \leq l \leq d.
\end{equation}
So $\Omega^1_{X^\dag / 0^\dag}$ is the $\cO_X$-module $\oplus_{i=1}^d \cO_X \frac{\td z_i}{z_i}$ modulo the above relations, which is a locally free sheaf with a global frame $\frac{\td z_1}{z_1},\dots,\frac{\td z_n}{z_n}$.

We set $\Omega^k_{X^\dag / 0^\dag} := \wedge^k \Omega^1_{X^\dag / 0^\dag}$. On any $n$-dimensional component $X_i$ of $X$, any $\alpha \in \Omega^k_{X^\dag / 0^\dag}$ has a unique representative $\alpha|_{X_i}$ given by a holomorphic $k$-form with at most log poles on $X_j \cap \text{Sing}\, X$. The compatibility condition is that if $X_i$ and $X_j$ has an $(n-1)$-dimensional intersection $D_{ij}$, then the restriction of $\alpha|_{X_i}$ and $\alpha|_{X_j}$ to $D_{ij}$ agrees when taking into account the relations \eqref{relation of log differential}.

Writing $\theta_i := z_i\frac{\pa}{\pa z_i}$ as vector fields on $\text{Spec}( \C[z])$, we can express $\Theta^1_{X^\dag / 0^\dag}$ as the $\cO_X$-submodule
\begin{equation} \label{relation of log derivation}
    \left\{\sum_{i=1}^d b_i \theta_i : b_l = \sum_{i=1}^n b_i a_{il}, \forall n+1 \leq l \leq d \right\}
\end{equation}
of $\Theta^1_{X^{\dag}}=\oplus_{i=1}^d \cO_X \theta_i$, and we set $\Theta^k_{X^\dag / 0^\dag} := \wedge^k \Theta^1_{X^\dag / 0^\dag}$. Any $v \in \Theta^1_{X^\dag / 0^\dag}$ restricts to a log derivation on each $n$-dimensional component of $X$.

Let $\Omega := \wedge_{i=1}^n \frac{\td z_i}{z_i}$. Then it defines a global nowhere vanishing log differential form in $\Omega^n_{X^\dag / 0^\dag}$ on $X$. Hence $X^\dag$ is in fact log Calabi-Yau.
As usual, contraction with $\Omega$ defines a map
\begin{equation*}
    \Theta^k_{X^\dag / 0^\dag} \to \Omega^{n-k}_{X^\dag / 0^\dag}, \quad \theta \mapsto \theta \vdash \Omega.
\end{equation*}

The following is the definition of a logarithmic version of Landau-Ginzburg models, though we will restrict ourselves only to a narrow class.
\begin{definition}[Logarithmic Landau-Ginzburg model]\label{log LG model}
    A \emph{logarithmic Landau-Ginzburg model} (abbrev. log LG model) is a triple $(X^\dag, \phi,  f)$, where $X^\dag$ is a log Calabi-Yau analytic space with underlying space $X$, $\phi:X^\dag \to 0^\dag$ is a log smooth map from $X^\dag$ to the standard $P$-log point $0^\dag$ associated to some toric monoid $P$, and $f$ is a holomorphic function (called the \emph{superpotential}) on $X$ such that when $\td f$ is viewed as a relative log differential, the critical set
    \begin{align*}
        \text{Crit}(f) := \{z\in \C^d : \td f(z) = 0\}
    \end{align*}
    is compact.
\end{definition}

\begin{remark}
	In Section \ref{constructions}, we are taking the toric monoid $P=\mathbb{N}$. In Section \ref{sec:deformation} and later sections when we consider universal unfoldings, we may need to consider a more general $P$. 
\end{remark}

\begin{remark}
    The definition of a log LG model may be generalized to the case when the log structure on $X$ has singularities, but we do not pursue it here.
\end{remark}

The function $f =z_1+\cdots+z_d$ defines a superpotential on $X$. As will be shown in the next section, $\td f$ has only isolated zeros, hence the triple $(X^\dag, \phi,  f)$ defines a log LG model; we regard it as a mirror of the projective toric manifold $Y$.

\subsection{The state space isomorphism}\label{state space isom}

Let us list some important complexes of sheaves.
Note that $\td f \wedge$ and the holomorphic de Rham differential $\pa$ are well-defined differentials on $\Omega^\bullet_{X^\dag / 0^\dag}$. The first complex is
\begin{equation*}
    (\Omega^\bullet_{X^\dag / 0^\dag}, \td f \wedge),
\end{equation*}
where $\td f$ is viewed as a global log differential.
By conjugating with $\Omega$, we get a second complex
\begin{equation*}
    (\Theta^\bullet_{X^\dag / 0^\dag},\{f,-\}).
\end{equation*}
Note that $\{f,-\}$ acts as a derivation, hence the cohomology group has a product structure.

Let $u$ be a formal variable. Then we can consider the complex
\begin{equation*}
     (\Omega^\bullet_{X^\dag / 0^\dag}[[u]], \td f \wedge + u \pa),
\end{equation*}
where $[[u]]$ denotes formal power series in $u$.
Again by conjugating with $\Omega$, we obtain another complex
\begin{equation*}
     (\Theta^\bullet_{X^\dag / 0^\dag}[[u]], \{f,-\} + u \pa).
\end{equation*}
The following result can be viewed as a mirror theorem \emph{without} quantum corrections.

\begin{theorem}[State space isomorphism] \label{state space isomorphism}
    Let $(X^\dag,\phi, f)$ be the log LG model obtained from a projective toric manifold $Y$. Then for $i>0$, we have
    \begin{equation*}
        \mathbb{H}^i(\Theta^\bullet_{X^\dag / 0^\dag},\{f,-\}) = 0;
    \end{equation*}
    while for $i=0$, we have the following ring isomorphism
    \begin{equation*}
        \mathbb{H}^0(\Theta^\bullet_{X^\dag / 0^\dag},\{f,-\}) \iso H^*(Y,\C).
    \end{equation*}
\end{theorem}
\begin{proof}
        We will show that
    \begin{equation}\label{cohomology computations}
        \mathbb{H}^i(\Omega^\bullet_{X^\dag / 0^\dag}, \td f \wedge) \iso
        \left\{
            \begin{array}{cc}
               0,  &  i<n;\\
               \C[z]/(P(\Sigma)+SR(\Sigma))\cdot \Omega,  & i=n.
            \end{array}
        \right.
    \end{equation}
    Using the conjugation by $\Omega$, the theorem then follows from the well-known result on the cohomologies of compact toric manifolds.
    In equation \eqref{cohomology computations}, $P(\Sigma)$ denotes the ideal in $\C[z]$ generated by $n$ elements
    \begin{equation*}
        \sum_{i=1}^d \langle \rho_i,e^1\rangle z_i, \cdots, \sum_{i=1}^d \abracket{\rho_i,e^n} z_i,
    \end{equation*}
    where $e^1,\cdots,e^n$ is the basis of the dual lattice $M = \check{N} = \text{Hom}(N,\mathbb{Z})$ dual to $e_1,\cdots,e_n$.
    When written as a log differential,
    \begin{equation*}
        \td f = z_1\cdot \frac{\td z_1}{z_1} + \cdots + z_d\cdot \frac{\td z_d}{z_d}.
    \end{equation*}
    By the relations (\ref{relation of log differential}), we have
    \begin{align*}
        \td f
        =& (z_1+\sum_{l=n+1}^d a_{l1}z_l)\frac{\td z_1}{z_1} + \cdots +(z_n+\sum_{l=n+1}^d a_{ln}z_l)\frac{\td z_n}{z_n} \\
        =& \sum_{i=1}^d\langle \rho_i,e^1 \rangle z_i \frac{\td z_1}{z_1} + \cdots +\sum_{i=1}^d\abracket{\rho_i,e^n}z_i \frac{\td z_n}{z_n}.
    \end{align*}
     Since $X$ is affine, the hypercohomology has the form 
     $$\mathbb{H}^i(\Omega^\bullet_{X^\dag / 0^\dag}, \td f \wedge) \iso
         H^i(\Omega^\bullet_{X^\dag / 0^\dag}(X), \td f \wedge).$$
    To conclude that $ \mathbb{H}^i(\Omega^\bullet_{X^\dag / 0^\dag}, \td f \wedge)$ has the desired property, it is sufficient to show that
    \begin{equation*}
        \sum_{i=1}^d \langle \rho_i,e^1 \rangle z_i, \cdots, \sum_{i=1}^d\abracket{\rho_i,e^n}z_i
    \end{equation*}
    form a regular sequence in $\C[z]/SR(\Sigma)$.
    This is true if the origin is the only common zero of the sequence, which in turn can be verified on each component of $X$.
    Let $X_1$ be the $n$-dimensional component with $z_i = 0$ when $n+1 \leq i \leq d$. Then on $X_1$, the sequence becomes $z_1,\cdots,z_n$, so indeed the origin is the only common zero.
    The same holds on any other component of $X$ since $\Sigma$ is a smooth fan.
\end{proof}

Since $H^*(\Theta^\bullet_{X^\dag / 0^\dag},\{f,-\})$ is concentrated at degree $0$, we have the following \emph{Hodge-to-de Rham degeneration property}:
\begin{theorem}[Hodge-to-de Rham degeneration] \label{Hodge-to-de Rham}
    The spectral sequence associated to the $u$-adic filtration of the complex $(T^*_{X^\dag / 0^\dag}[[u]],\{f,-\}+ u \pa)$ degenerates at the $E_1$ page; in particular, we have the isomorphism
    \begin{equation*}
        H^*(\Theta^\bullet_{X^\dag / 0^\dag}[[u]],\{f,-\}+u \pa) \iso H^*(\Theta^\bullet_{X^\dag / 0^\dag},\{f,-\})[[u]].
    \end{equation*}
\end{theorem}

\begin{remark}
    Results on complexes of log differential forms can be translated into results on complexes of log derivations by conjugation with $\Omega$, and vice versa. A minor difference is that it is possible to define a ring structure on the complexes of log derivations.
\end{remark}

\subsection{A homotopy formula}\label{homotopy formula}

Let $\A^{0,j}$ be the sheaf of smooth differential $(0,j)$-forms on $X$.
Concretely, $\alpha \in \A^{0,j}(U)$ is given by a smooth $(0,j)$-form $\alpha_i$ on each $U \cap X_i$ such that the pull-back of $\alpha_i$ and $\alpha_j$ to $U \cap X_i \cap X_j$ coincide.
Define
\begin{equation*}
    \A^{i,j}_{\log} := \A^{0,j} \otimes_{\mathcal{O}_X} \Omega^i_{X^\dag/0^\dag},\quad
    \PV^{i,j}_{\log} := \A^{0,j} \otimes_{\mathcal{O}_X} \Theta^i_{X^\dag/0^\dag}.
\end{equation*}
By Theorem \ref{Dolbeault_resolution} and the fact that $\Omega^i_{X^\dag/0^\dag}$ and $\Theta^i_{X^\dag/0^\dag}$ are locally free sheaves on $X$, we have the following corollary:
\begin{corollary}
    The following two complexes are exact:
    \begin{align*}
        0& \to \Omega^p_{X^\dag / 0^\dag} \to \A^{p,0}_{\log} \stackrel{\bar\pa}{\longrightarrow} \A^{p,1}_{\log} \stackrel{\bar\pa}{\longrightarrow} \A^{p,2}_{\log} \stackrel{\bar\pa}{\longrightarrow} \cdots, \\
        0& \to \Theta^p_{X^\dag / 0^\dag}\to \PV^{p,0}_{\log} \stackrel{\bar\pa}{\longrightarrow} \PV^{p,1}_{\log} \stackrel{\bar\pa}{\longrightarrow} \PV^{p,2}_{\log} \stackrel{\bar\pa}{\longrightarrow} \cdots.
    \end{align*}
\end{corollary}

Let $\A^{i,j}_{\log}(X)$ and $\PV^{i,j}_{\log}(X)$ be the global sections of these sheaves. We define $$\A_{\log}(X) := \bigoplus_{i,j} \A^{i,j}_{\log}(X),\quad 
\PV_{\log}(X) := \bigoplus_{i,j} \PV^{i,j}_{\log}(X),$$
and denote by $\A_{\log,c}(X)$ and $\PV_{\log,c}(X)$ the subspaces of the corresponding spaces that consist of elements with only compact support. It is clear that they are all super-commutative algebras over the algebra of smooth functions on $X$. In the sequel, if $\alpha$ is in $\A^{i,j}_{\log}(X)$ or $\PV^{i,j}_{\log}(X)$, we will call it a differential form or polyvector field of \emph{bidgree $(i,j)$}, and define $|\alpha| = i+j$ for $\alpha \in \A^{i,j}_{\log}(X)$ and $|\alpha| = i-j$ for $\alpha \in \PV^{i,j}_{\log}(X)$.
By abuse of notation, we will use $\bar\pa_f$ to denote the operator $\bar\pa + \td f \wedge$ when it acts on differential forms and the operator $\bar\pa + \{f,-\}$ when it acts on polyvector fields.
\begin{corollary}
    We have
    \begin{align*}
        H(\Omega^\bullet_{X^\dag / 0^\dag},\td f) \cong& H( \A_{\log}(X), \bar\pa_f),\\
        H(\Theta^\bullet_{X^\dag / 0^\dag},\{f,-\}) \cong& H( \PV_{\log}(X), \bar\pa_f).
    \end{align*}
\end{corollary}

If we denote by $Q_f$ the operator $\bar\pat_f + u\pat$, we have a similar result:
\begin{corollary}
    We have
    \begin{align*}
        H(\Omega^\bullet_{X^\dag / 0^\dag}[[u]],\td f \wedge +u\pat) \cong& H( \A_{\log}(X)[[u]], Q_f),\\
        H(\Theta^\bullet_{X^\dag / 0^\dag}[[u]],\{f,-\}+u\pat) \cong& H( \PV_{\log}(X)[[u]], Q_f).
    \end{align*}
\end{corollary}

We will do integration on $X$ to define the so-called \emph{higher residue pairing}. Since $X$ is non-compact, we need a homotopy formula to transform an arbitrary smooth differential form into one with only compact support.
The following construction is a variation of that in \cite{LLS}.

Let $\rho = \rho(|z_1|^2,\cdots,|z_d|^2)$ be a smooth (real-valued) cut-off function with compact support such that $\rho \equiv 1$ in a neighbourhood of $0$.
For $1 \leq k \leq n$, write $g_k := \sum_{i=1}^d \langle \rho_i,e^k \rangle z_i = z_k + \sum_{l=n+1}^d a_{kl}z_l$, and define
\begin{equation}\label{V_f}
    V_f := \frac{1}{\sum_{k=1} g_k \bar g_k} \left( \sum_{k=1}^n \bar g_k \theta_k + \sum_{l=n+1}^d (\sum_{k=1}^n a_{kl} \bar g_k) \theta_l \right),
\end{equation}
where $\bar g_k \in \mathcal{A}^{0,0}(X)$. If we define a global frame 
\begin{equation}\label{eqn:defining_theta_tilde}
\tilde{\theta}_i:= \theta_i + \sum_{l=n+1}^d a_{il} \theta_l 
\end{equation}
of $\Theta_{X^\dag/0^\dag}^1$, we have $V_f = \frac{1}{\sum_{k=1} g_k \bar g_k} \left(\sum_{k=1}^n \bar g_k \tilde{\theta}_k \right)$. The zeros of $\sum_{k=1} g_k \bar g_k$ contain only the origin since $\Sigma$ is smooth. As $V_f$ satisfies the condition of (\ref{relation of log derivation}), $V_f$ is a relative log derivation on $X-0$ and it defines an operator acting on $\PV_{\log}(X-0)$.

\begin{lemma} \label{V_f identity}
\begin{equation*}
    [\{f,-\},V_f] = 1 \qquad \text{on } X-0.
\end{equation*}
\end{lemma}
\begin{proof}
    Let $B = (b_{ij})$ be the $n\times d$ matrix such that $b_{ij} = \delta_{ij}$ for $j\leq n$ and $b_{ij} = a_{ij}$ for $j >n+1$, and let $X_i$ be the component of $X$ where the generating rays of the corresponding cone are given by $\rho_{i_1},\cdots,\rho_{i_n}$.
    Denote by $C_i$ the matrix whose $k$-th column is the $i_k$-th column of $B$.
    Note that $C_i$ is invertible since $\Sigma$ is a smooth fan.
    We can write
    \begin{align*}
        \td f = \sum_{k=1}^n g_k \frac{\td z_k}{z_k} = \left(\frac{\td z_{i_1}}{z_{i_1}}, \cdots, \frac{\td z_{i_n}}{z_{i_n}}\right) (C_i)^{-1} (g_1,\cdots,g_n)^T.
    \end{align*}
    By conjugating with $\Omega$ and using the frame $(\tilde{\theta}_{i_1},\cdots , \tilde{\theta}_{i_n})  = (\tilde{\theta}_1,\cdots, \tilde{\theta}_n ) (C_i^T)^{-1}$,
    \begin{align} \label{eq: {f,-}}
        \{f,-\} = \left(\frac{\pa}{\pa \tilde{\theta}_{i_1}}, \cdots, \frac{\pa}{\pa \tilde{\theta}_{i_n}}\right) (C_i)^{-1} (g_1,\cdots,g_n)^T,
    \end{align}
    where $[\frac{\pa}{\pa \tilde{\theta}_{i_j}},\tilde{\theta}_{i_s} \wedge ] = \delta_{js}$ as operators. On the other hand, on $X_i - 0$ we can write
    \begin{align*}
        V_f = \frac{1}{\sum_{k=1} g_k \bar g_k} (\bar g_1,\cdots, \bar g_k) C_i (\tilde{\theta}_{i_1},\cdots, \tilde{\theta}_{i_n})^T.
    \end{align*}
    Hence $ [\{f,-\},V_f] = \frac{1}{\sum_{k=1} g_k \bar g_k} (\bar g_1,\cdots, \bar g_k) C_i (C_i)^{-1} (g_1,\cdots,g_n)^T = 1$.
\end{proof}

\begin{proposition} \label{prop: homotopy formula}
    The inclusion of complexes
    \begin{align*}
        (\PV_{\log,c}(X),\bar\pa_f) \hookrightarrow (\PV_{\log}(X),\bar\pa_f)
    \end{align*}
    is a quasi-isomorphism.
\end{proposition}

\begin{proof}
    Define another two operators:
    \begin{equation*}
        T_{\rho} := \rho + (\bar{\pat} \rho)V_f \frac{1}{1+[\bar{\pat},V_f]},\quad R_{\rho} := (1-\rho)V_f \frac{1}{1+[\bar{\pat},V_f]}.
    \end{equation*}
    It is clear that they both act on $\PV_{\log}(X)$.
    As in \cite{LLS}, we have the following identity
    \begin{align*}
        [\bar\pa_f,R_{\rho}] = 1 - T_{\rho}.
    \end{align*}
    So $T_\rho$ gives the quasi-isomorphism.
\end{proof}

When the formal parameter $u$ is present, we have a similar result:
\begin{proposition} \label{homotopy formula2}
    The inclusion of complexes
    \begin{align*}
        (\PV_{\log,c}(X)[[u]],Q_f) \hookrightarrow (\PV_{\log}(X)[[u]],Q_f)
    \end{align*}
    is a quasi-isomorphism.
\end{proposition}
\begin{proof}
    Defining $Q:=\bar\pat+u\pat$ and
    \begin{align*}
        T_{\rho}^u := \rho + (Q \rho)V_f \frac{1}{1+[Q,V_f]},\quad
        R_{\rho}^u := (1-\rho)V_f \frac{1}{1+[Q,V_f]},
    \end{align*}
    one finds
    \begin{align*}
        [Q_f,R_{\rho}^u] = 1 - T_{\rho}^u.
    \end{align*}
    Now $T_\rho^u$ gives the desired quasi-isomorphism. 
\end{proof}
Similarly we can define $T_{\rho}^u$ and $R_{\rho}^u$ acting on $(\A_{\log}(X)[[u]],Q_f)$. 

\begin{corollary} \label{homotopy formula for PV}
    The inclusions of complexes
    \begin{align*}
        (\A_{\log,c}(X),\bar\pa_f) \hookrightarrow& (\A_{\log}(X),\bar\pa_f),\\
        (\A_{\log,c}(X)[[u]],Q_f) \hookrightarrow& (\A_{\log}(X)[[u]],Q_f)
    \end{align*}
    are quasi-isomorphisms.
\end{corollary}

\subsection{Integration and the higher residue pairing}\label{integration}

Now let us define the integral of $\phi \in \A^{p,q}_{\log,c}(X)$ on $X$.
Since $\bigwedge_{i=1}^n (\frac{\td z_i}{z_i} \wedge \td \bar z_i)$ is locally integrable around $0 \in \C^n$, one can define
\begin{align*}
    \int_X \phi :=
    \left\{
        \begin{array}{cc}
           \sum_i \int_{X_i} \phi_i,  &  (p,q) = (n,n);\\
           0,  & (p,q) \neq (n,n).
        \end{array}
    \right.
\end{align*}
Here the summation is over all the $n$-dimensional components $X_i$ of $X$ and $\varphi_i$ is the unique representative of $\varphi$ on $X_i$.
The following lemma is of fundamental importance in the sequel.

\begin{lemma} \label{Vanishing of integral}
    Let $\phi \in \A^{n-1,n}_{\log,c}(X)$ and $\psi \in \A^{n,n-1}_{\log,c}(X)$. Then we have
    \begin{align*}
        \int_X \pa \phi = \int_X \bar\pa \psi = 0.
    \end{align*}
\end{lemma}

\begin{proof}
Denote by $\text{Sing}(X)$ the singular part of $X$ and define $D_i := X_i \cap \text{Sing}(X)$.
Denote by $E_i$ the set $\{j : \text{dim}(X_i \cap X_j)=n-1\}$ and write $D_{ij} = X_i \cap X_j$. Then $D_i = \cup_{j\in E_i} D_{ij}$ is a union of $(n-1)$-dimensional subspaces of $X_i$.
Note that $\alpha= (\alpha_i) \in \A^{p,q}_{\log,c}(X)$ implies that for any $i$, $\alpha_i \in \A^{0,q}(X_i,\Omega_{X_i}^p(\log D_i))$.
One can now define the Poincar\'e residue of $\alpha_i$ on $D_{ij}$, and in particular,
\begin{align*}
    \text{Res}_{D_{ij}} \alpha_i \in \A^{p-1,q}_{D_{i,j}}\left(\log \left(D_{i,j}\cap \bigcup_{q\in E_i,q\neq j} D_{i,q}\right)\right).
\end{align*}
In \cite[Section 2]{LRW}, the following formulae are proved:
\begin{align*}
    \int_{X_i} \pa \alpha_i \wedge \beta_i =& (-1)^{p+q+1} \int_{X_i} \alpha_i \wedge \pa \beta_i,\\
    \int_{X_i} \bar\pa \alpha_i \wedge \gamma_i =& (-1)^{p+q+1} \int_{X_i} \alpha_i \wedge \bar\pa \gamma_i - 2\pi \sqrt{-1} \sum_{j\in E_i} \int_{D_{ij}} \text{Res}_{D_{ij}} \alpha_i \wedge \iota^*\gamma_i,
\end{align*}
where  $\alpha_i \in \A_{X_i}^{p,q}(\log D_i), \beta_i \in \A^{n-p-1,n-q}(X_i),\gamma_i \in \A^{n-p,n-q-1}(X_i)$ and $\iota: D_{ij} \to X_i$ is the inclusion map.
Note that these formulae can also be deduced from Theorem \ref{Stokes formula 2}.

For $\alpha = \phi, \beta = 1$, where $1$ denotes the function that is equal to $1$ on each component, we get $\int_X \pa \phi =0$.
For $\alpha = \psi, \gamma = 1$, we denote by $E(\Delta)$ the set $\{(i,j) : i<j,\text{dim}(X_i \cap X_j)=n-1\}$, and then we have
\begin{align*}
    \int_X \bar\pa \psi
    =& -2\pi \sqrt{-1} \sum_i \sum_{j\in E_i} \int_{D_{ij}} \text{Res}_{D_{ij}} \psi_i \\
    =& -2\pi \sqrt{-1} \sum_{(i,j)\in E(\Delta)} \int_{D_{ij}} (\text{Res}_{D_{ij}} \psi_i + \text{Res}_{D_{ji}} \psi_j) \\
    =& 0.
\end{align*}
The last equality holds because $\text{Res}_{D_{ij}} \psi_i + \text{Res}_{D_{ji}} \psi_j=0$, which follows from the definition of $\A^{p,q}_{\log,c}(X)$.
\end{proof}

Integration on $X$ induces a pairing on $\A_{\log,c}(X)$ defined as $\abracket{\phi,\psi}_{\A} = \int_X \phi \wedge \psi$.
By applying Lemma \ref{Vanishing of integral}, we get the following
\begin{corollary}
    For $\phi \in \A^{p,q}_{\log,c}(X), \psi \in \A_{\log,c}(X)$, we have
    \begin{align*}
        \abracket{\pa\phi,\psi}_{\A} =& (-1)^{|\phi|+1} \abracket{\phi,\pa\psi}_{\A}, \\
         \abracket{\bar\pa\phi,\psi}_{\A} =& (-1)^{|\phi|+1} \abracket{\phi,\bar\pa\psi}_{\A}.
    \end{align*}
\end{corollary}

With the global nowhere vanishing form $\Omega$ at hand, we can define the trace of polyvector fields. Given $\alpha \in \PV_{\log,c}(X)$, define its \emph{trace} to be
\begin{align*}
    \Tr \alpha := \int_X (\alpha \vdash \Omega) \wedge \Omega.
\end{align*}
Note that $\Tr \alpha$ vanishes unless $\alpha$ is of bidegree $(n,n)$.
\begin{proposition} \label{Vanishing of trace}
    For $\alpha \in \PV_{\log,c}(X)$, we have
    \begin{align*}
        \Tr (\bar\pa \alpha) = \Tr (\pa \alpha) = \Tr \{f, \alpha\} = 0.
    \end{align*}
\end{proposition}
\begin{proof}
    We have $\Tr (\bar\pa \alpha) = \int_X (\bar\pa \alpha  \vdash \Omega) \wedge \Omega = \int_X \bar\pa ((\alpha \vdash \Omega) \wedge \Omega) = 0$. $\Tr (\pa \alpha)$ and $\Tr \{f, \alpha\}$ vanish because $\pa \alpha$ and $\{f, \alpha\}$ have no nonzero components of bidegree $(n,n)$.
\end{proof}

Similarly, this trace map induces a pairing on polyvector fields by defining $$\abracket{\alpha,\beta}_{\PV} =: \Tr (\alpha \cdot \beta)$$
for $\alpha,\beta \in \PV_{\log,c}(X)$. An explicit computation in coordinates gives the following lemma.

\begin{lemma}
    Assume $\alpha \in \PV^{p,q}_{\log}(X), \beta \in \PV^{n-p,n-q}_{\log}(X)$. Then
    \begin{align*}
        ((\alpha \cdot \beta) \vdash \Omega) \wedge \Omega = (-1)^{n(q+1)+p(n+1)} (\alpha \vdash \Omega) \wedge (\beta \vdash \Omega).
    \end{align*}
\end{lemma}

\begin{proposition} \label{Adjoints of operators}
    For any $\alpha,\beta \in \PV_{\log,c}(X)$, we have
    \begin{align*}
        \abracket{\bar\pa \alpha, \beta}_{\PV} =& -(-1)^{|\alpha|}\abracket{\alpha,\bar\pa\beta}_{\PV}, \\
        \abracket{\{f,\alpha\}, \beta}_{\PV} =& -(-1)^{|\alpha|}\abracket{\alpha,\{f,\beta\}}_{\PV}, \\
        \abracket{\pa \alpha, \beta}_{\PV} =& (-1)^{|\alpha|}\abracket{\alpha,\pa\beta}_{\PV}.
    \end{align*}
\end{proposition}
\begin{proof}
    For the third identity, we can assume for simplicity that $\alpha \in \PV^{p+1,q}_{\log,c}(X), \beta \in \PV^{n-p,n-q}_{\log,c}(X)$, and then we have
    \begin{align*}
        \abracket{\pa \alpha, \beta}_{\PV}
        =& \int_X ((\pa \alpha \cdot \beta) \vdash \Omega) \wedge \Omega \\
        =& (-1)^{n(q+1)+p(n+1)} \int_X (\pa \alpha \vdash \Omega) \wedge (\beta \vdash \Omega) \\
        =& (-1)^{n(q+1)+p(n+1)+n-p-1+q+1} \int_X (\alpha \vdash \Omega) \wedge (\pa\beta \vdash \Omega) \\
        =& (-1)^{n(q+1)+p(n+1)+n-p-1+q+1 + n(q+1)+(p+1)(n+1)} \int_X ((\alpha \wedge  \pa\beta) \vdash \Omega) \wedge \Omega \\
        =& (-1)^{p+1+q} \Tr (\alpha \cdot \pa\beta) \\
        =& (-1)^{|\alpha|} \abracket{\alpha,\pa\beta}_{\PV}.
    \end{align*}
    The first two identities can be proved similarly, or simply by using Proposition \ref{Vanishing of trace} and the fact that $\bar\pa$ and $\{f,-\}$ are derivations.
\end{proof}

By Theorems \ref{Vanishing of trace} and \ref{Adjoints of operators}, the trace map and the pairing both descend to cohomologies.
Applying Corollary \ref{homotopy formula for PV}, we can define on the (formally completed) \emph{Brieskorn lattice} 
$$\mathcal{H}_f := H(\PV_{\log}(X)[[u]],Q_f)$$
a trace map and a pairing: 
for $\alpha(u), \beta(u) \in \mathcal{H}_f$,
\begin{align*}
    \Res_f: \mathcal{H}_f \to \C[[u]], \quad
    \qquad \Res_f(\alpha(u)):= \Tr(T_\rho^u \alpha(u));
\end{align*}
and
\begin{align*}
    \KK_f: \mathcal{H}_f \times \mathcal{H}_f \to \,\C[[u]],
    \qquad \KK_f(\alpha(u),\beta(u)) := \Tr(T_\rho^u \alpha(u) \wedge \overline{T^u_\rho \beta(u)}),
\end{align*}
where $\overline{\beta(u)} := \beta(-u)$.

Note that $\Res_f$ and $\KK_f$ can be extended to $H(\PV_{\log}(X)((u)),Q_f)$ linearly in $u$, which is isomorphic to $\mathcal{H}_f \otimes_{\C[[u]]} \C((u))$ by Theorem \ref{Hodge-to-de Rham} and will be denoted by $\mathcal{H}_{f,\pm}$.
In the sequel, we will call $\Res_f$ and $\KK_f$ the \emph{higher residue map} and the \emph{higher residue pairing} respectively.

\subsection{Grading structure and good basis}\label{grading}

There exists natural grading structures on the complexes and cohomologies. To be compatible with the cohomology grading of $H^*(Y,\C)$,
we put $$\deg(z_i) = 2, \quad \forall 1 \leq i \leq d$$ so that $\deg(f) = 2$.
In this way, it is natural to put $\deg(\td z) = 2, \deg(\bar z) = \deg(\td \bar z) = -2$.
We also put
\begin{align*}
    \deg(\tilde{\theta}_i) = 2,
\end{align*}
where $\tilde{\theta}_i$ is defined in equation \eqref{eqn:defining_theta_tilde} from $\theta_i$'s. Finally, we put $\deg(u) = 2$.

We can pack up the above degree assumption by defining a grading operator $E_f: \PV_{\log}(X)((u)) \to \PV_{\log}(X)((u))$ by
\begin{align*}
    E_f := 2u\frac{\pa}{\pa u} + 2\sum_{i=1}^d \left(z_i \frac{\pa}{\pa z_i} -\bar z_i \frac{\pa}{\pa \bar z_i} - \td \bar z_i \wedge \frac{\pa}{\pa \td \bar z_i}\right)  +2 \sum_{i=1}^n \tilde{\theta}_i \wedge \frac{\pa}{\pa \tilde{\theta}_i},
\end{align*}
where $\frac{\pa}{\pa \tilde{\theta}_i}$ is an odd operator satisfying $[\frac{\pa}{\pa \tilde{\theta}_i}, \tilde{\theta}_i] = 1$. 
Let $\Delta$ be the defining polytope of the toric manifold $Y$ and $\mathcal{C}^i(\Delta, \A^{0,*})$ be the sheaf defined in Appendix \ref{Dolbeault}.
The operators $z_i \frac{\pa}{\pa z_i}$, $\bar z_i \frac{\pa}{\pa \bar z_i}$ and $d \bar z_i \wedge \frac{\pa}{\pa \td \bar z_i}$ act on $\mathcal{C}^0(\Delta, \A^{0,*}) = \bigoplus_{i=1}^d \A^{0,*}(X_i)$ (on each component $X_i$, only non-zero $z_i$'s and $\bar z_i$'s in the summand act) and these actions pass on to the subspaces $\A^{0,*}(X)$ using Proposition \ref{Check resolution_3}, while the operator $\tilde{\theta}_i \wedge \frac{\pa}{\pa \tilde{\theta}_i}$ acts on $\Theta_{X^\dag/0^\dag}^*$.

Using relation (\ref{relation of log differential}) to write $\partial = \sum_{i=1}^d z_i \frac{\partial}{\partial z_i} \frac{\td z_i}{z_i} = \sum_{i=1}^n (z_i\frac{\pa}{\pa z_i} + \sum_{l=n+1}^d a_{li}z_l\frac{\pa}{\pa z_l}) \frac{\td z_i}{z_i}$. Then by contraction with $\Omega$ and equation(\ref{eq: {f,-}})
we have
\begin{align*}
    Q_f 
    =& \, \bar\partial + \{f,-\} + u\partial \\
    =& \sum_i^d \td \bar z_i \wedge \frac{\pa}{\pa \bar z_i} + \sum_{i=1}^n \left(g_i  + u \left(z_i\frac{\pa}{\pa z_i} + \sum_{l=n+1}^d a_{li}z_l\frac{\pa}{\pa z_l} \right)\right) \frac{\pa}{\pa \tilde{\theta}_i}.
\end{align*}
$Q_f$ preserves the $\Z$-grading induced by $E_f$ and thus
\begin{align*}
    [E_f, Q_f] = 0.
\end{align*}
We conclude that the action of $E_f$ descends to cohomologies.
In particular, the isomorphism in Theorem \ref{state space isomorphism} is an isomorphism of graded rings.

\begin{lemma} \label{weight property}
    The higher residue map $\Res_f$ and the higher residue pairing $\KK_f$ both have weight degree $-2n$.
\end{lemma}
\begin{proof}
    This is a straightforward calculation, using the fact that $\Tr \alpha, \alpha \in \PV_{\log, c}$ vanishes unless $\alpha$ has weight degree $2n$.
\end{proof}

Now let us describe a set of distinguished generators of $\mathcal{H}_f$.
As proved in \cite[Section 5.2]{Ful}, we can fix an order $\sigma_1,\cdots,\sigma_\mu$ of $n$-dimensional cones of $\Sigma$ such that if we let $\tau_i \subset \sigma_i,\ 1\leq i \leq \mu$ be the intersection of $\sigma_i$ with all cones in $\{\sigma_j : j>i,\ \dim(\sigma_i \cap \sigma_j) = n-1\}$, then we have
\begin{equation}\label{eqn:moving_lemma_condition}
    \text{if } \tau_i \text{ is contained in } \sigma_j, \text{ then } i \leq j.
\end{equation}
Given a cone $\tau$ in $\Sigma$, let $P(\tau)$ be the monomial $z_{i_1}\cdots z_{i_k}$ where $\rho_{i_1}, \cdots, \rho_{i_k}$ are the generating rays of $\tau$.
By the results in the same section of \cite{Ful}, $\{P(\tau_1), \cdots, P(\tau_\mu)\}$ forms a $\C$-basis of $H^0(\Theta^\bullet_{X^\dag / 0^\dag},\{f,-\})$.
By Theorem \ref{Hodge-to-de Rham}, they also form a set of generators of the  free $\C[[u]]$-module $\mathcal{H}_f$.
If we denote $\varphi_i := P(\tau_i)$, then $\varphi_1 = 1$ and we write $\varphi_\mu = z_1 \cdots z_n$ for convenience.

The following lemma is a refined version of the ``algebraic moving lemma'' in page 107 of \cite{Ful}.

\begin{lemma}[Moving Lemma] \label{algebraic moving lemma}
Let $\gamma_1 \subsetneq \sigma \subset \gamma_2$ be cones in $\Sigma$ and let $k =\dim(\sigma)$. Then there are cones $\sigma_i$ of dimension $k$ in $\Sigma$ with $\gamma_1 \subset \sigma_i$ and $\sigma_i \nsubseteq \gamma_2$ and integers $m_i$ such that
\begin{equation*}
    P(\sigma) = \sum_i m_i P(\sigma_i) \in \mathcal{H}_f.
\end{equation*}
\end{lemma}
\begin{proof}
	Without loss of generality, we assume that $\gamma_1, \sigma$ and $\gamma_2$ are generated by $\{\rho_1,\cdots,\rho_p\}$, $\{\rho_1,\cdots,\rho_k\}$ and $\{\rho_1,\cdots,\rho_q\}$ respectively, where $1\leq p < k \leq q \leq n$.
    For $p+1 \leq i \leq k$, $\tilde\theta_i = \theta_i + \sum_{l=n+1}^d a_{il} \theta_l \in \Theta^1_{X^\dag / 0^\dag}$ by (\ref{relation of log derivation}), so 
    we have
    \begin{equation*}
    \mathcal{H}_f \ni 0 = (\{f,-\}+ u\pa) (z_1\cdots \hat z_i \cdots z_k \tilde\theta_i) = z_1 \cdots z_k + \sum_{l=n+1}^d a_{il} z_l z_1\cdots \hat z_i \cdots z_k.
\end{equation*}
    Neglecting the terms $z_l z_1\cdots \hat z_i \cdots z_k$ such that the corresponding rays do not lie in a common cone, we get the desired result.
\end{proof}

\begin{theorem} \label{good basis}
    The basis described above is a good basis, namely, $\KK_f(\varphi_i, \varphi_j) \in \C$ for any $i, j$.
\end{theorem}
\begin{proof}
    If $\deg(\varphi_i)+\deg(\varphi_j) \leq 2n$, then $\KK_f(\varphi_i, \varphi_j) \in \C$ by Lemma \ref{weight property}.
    So we assume that $\deg(\varphi_i)+\deg(\varphi_j) > 2n$.
    If $\tau_i$ and $\tau_j$ are not contained in a common cone of $\Sigma$, then $\KK_f(\varphi_i, \varphi_j) = 0$ since on each component of $X$ either $\varphi_i$ or $\varphi_j$ vanishes.
    Thus we may further assume that $\tau_i$ and $\tau_j$ are contained in a common maximal cone $\sigma \in \Sigma$.
    In this case, the number of common variables of $\varphi_i$ and $\varphi_j$ is $m \geq 1$.
    Assuming $z_k$ is a common variable, then the Moving Lemma \ref{algebraic moving lemma} tells us that we can replace $z_k$ in $\varphi_i$ by a linear combination of variables not corresponding to rays in $\sigma$.
    By neglecting the terms whose variables and that of $\varphi_j$ do not correspond to rays in a common cone of $\Sigma$, we reduce to the case that $\varphi'_i$ and $\varphi'_j$ has $m-1$ common variables.
    By inductively reducing the number of common variables, we arrived at the case that $\varphi'_i$ and $\varphi'_j$ has no common variables and $\deg(\varphi'_i)+\deg(\varphi'_j) > 2n$. In this case, $\KK_f(\varphi'_i, \varphi'_j) = 0$ since the corresponding rays cannot lie in a common cone.
\end{proof}

\begin{proposition} \label{nondegeneracy}
    Let $\KK_f^{(0)}: H(\PV_{\log}(X),\bar\pa_f) \times H(\PV_{\log}(X),\bar\pa_f) \to \C$ be the leading term in the expansion in $u$ of the higher residue pairing $\KK_f$.
    Under the ring isomorphism in Theorem \ref{state space isomorphism}, $\KK_f^{(0)}$ coincides with the Poincar\'e pairing of $Y$ up to a nonzero constant depending only on the dimension $n$.
    In particular, $\KK_f^{(0)}$ is non-degenerate.
\end{proposition}
\begin{proof}
    In fact, $\KK_f^{(0)}$ is the trace of the product of two classes in $ H(\PV_{\log}(X),\bar\pa_f)$, and is hence defined in the same way as the Poincar\'e (or cup product) pairing of $Y$.
    By Theorem \ref{state space isomorphism} and Lemma \ref{weight property}, it remains to show the number $\Res_f(z_1\cdots z_n) \in \C$ is nonzero.
    The following sketched calculation is parallel to that of the proof of Proposition 2.5 in \cite{LLS}.
    First, one can show by similar calculations that
    \begin{equation*}
        V_f [\bar\pa, V_f]^{n-1} = (-1)^{n(n-1)/2} (n-1)! \sum_{i=1}^n (-1)^{i-1} \frac{\bar z_i}{|z|^{2n}} \td \bar z_1 \wedge \cdots \widehat{\td \bar z_i} \wedge \cdots \td \bar z_n \otimes \theta_1 \cdots \theta_n
    \end{equation*}
    and it is $\bar\pa$-closed. Then, by the residue theorem, we have
    \begin{align*}
         & \Res_f(z_1\cdots z_n)\\
        =& (-1)^{n-1} \int_{X_1} \left(\bar \pa \rho V_f [\bar\pa, V_f]^{n-1} (z_1\cdots z_n) \vdash \frac{\td z_1}{z_1} \wedge \cdots  \wedge \frac{\td z_n}{z_n}\right) \wedge \left(\frac{\td z_1}{z_1} \wedge \cdots \wedge \frac{\td z_n}{z_n}\right) \\
        =& (-1)^{n-1} \int_{X_1} \left(\bar \pa \rho V_f [\bar\pa, V_f]^{n-1} \vdash \frac{\td z_1}{z_1} \wedge \cdots \wedge\frac{\td z_n}{z_n}\right) \wedge (\td z_1 \wedge \cdots \wedge \td z_n) \\
        =& (-1)^{n-1} (n-1)!\\
        & \quad \cdot\int_{X_1} \td \left(\rho \sum_{i=1}^n (-1)^{i-1} \frac{\bar z_i}{|z|^{2n}} (\td \bar z_1 \wedge \cdots \wedge \widehat{\td \bar z_i} \wedge \cdots \wedge \td \bar z_n) \wedge (\td z_1 \wedge \cdots \td z_n)\right)\\
        =& (-1)^n (n-1)! \int_{|z|^2 = R} \left(\sum_{i=1}^n (-1)^{i-1} \frac{\bar z_i}{|z|^{2n}} \td \bar z_1 \wedge \cdots \widehat{\td \bar z_i} \wedge \cdots \td \bar z_n) \wedge (\td z_1 \wedge \cdots \td z_n)\right) \\
        =& (-2 \pi \sqrt{-1})^n (n-1)!.
    \end{align*}

\end{proof}

\section{Deformation theory of log LG models} \label{sec:deformation}

\subsection{Unfolding}\label{unfolding}

An unfolding of our log LG model consists of two parts -- one corresponds to the smoothing of the target space $X$ and the other corresponds to deformation of the superpotential $f$. The former is described by a construction in \cite{GHK}.

Let $P := \overline{NE}(Y)$ be the cone of effective curves in $Y$ with $P^{gp} = A_1(Y,\Z)$. We consider the following exact sequence
\begin{align*}
    0 \to A_1(Y,\Z) \to \Z^{|\Sigma(1)|} \stackrel{s}{\to} N \to 0,
\end{align*}
where $s$ is the map sending the basis element $t_{\rho}$, $\rho \in \Sigma(1)$ to $\rho$ itself.
Let $\pi: \Z^{|\Sigma(1)|} \to  A_1(Y,\Z)$ be any splitting of the above sequence and let $\tilde\varphi$ be the unique $ \Z^{|\Sigma(1)|}$-valued $\Sigma$-piecewise linear function given by sending $\rho$ to $t_{\rho}$.
Define
\begin{align*}
    \varphi = \pi \circ \tilde\varphi: N \to A_1(Y,\Z).
\end{align*}
Then by \cite[Lemma 1.13]{GHK}, the map $\varphi: |\Sigma| \to P^{gp} = A_1(Y,\Z)$ is a strictly $P$-convex $\Sigma$-piecewise linear function.
Consider
\begin{align*}
    P_{\varphi} := \{(n,\varphi(n)+p) : n \in |\Sigma|, p \in P\} \subset N \times P^{gp}.
\end{align*}
The convexity of $\varphi$ implies that $P_{\varphi}$ is a monoid.
The natural inclusion $P \hookrightarrow P_{\varphi}$ given by $p \mapsto (0,p)$ induces a flat morphism
\begin{align*}
    \Phi: \Spec (\C[P_{\varphi}]) \to \Spec (\C[P]).
\end{align*}

Since $Y$ is projective, $\overline{NE}(Y)$ and the nef cone $Nef(Y)$ of $Y$ are dual strictly convex rational polyhedral cones of full dimension by \cite[Theorem 6.3.22]{CLS}.
There is a bijection between the interior lattice points $q \in Nef(Y)$ and strictly convex $\Sigma$-piecewise linear functions, up to adding a linear function, with only integral slopes on $|\Sigma|$.
Let $\abracket{-,-}$ be the pairing between $Nef(Y)$ and $\overline{NE}(Y)$. Then $q$ defines a map of monoids
\begin{align*}
    q: \overline{NE}(Y) \to \N, p \mapsto \abracket{q,p}.
\end{align*}
Since $\overline{NE}(Y)$ is strictly convex and $q$ is in the interior of $Nef(Y)$, the preimage of $0 \in \N$ contains only $0 \in \overline{NE}(Y)$ and hence the function $q \circ \varphi: |\Sigma| \to \R$ is a strictly convex $\Sigma$-piecewise linear function.
In fact, the function $q \circ \varphi$ is exactly the one $q\in Nef(Y)$ represents.

Now fix $q \in Nef(Y)$ such that it represents the function $\psi$ in Section \ref{constructions}. We have the following commutative diagram
\begin{align*}
    \xymatrix
    {
    Q & P_{\varphi} \ar[l] \\
    \N \ar[u] & P \ar[l] \ar[u]
    }
\end{align*}
where $Q := \{(n,q \circ \varphi(n)+p) : n \in |\Sigma|, p \in \N\}$. Therefore we have
\begin{align*}
    \xymatrix
    {
    \Spec(\C[Q]) \ar[d] \ar[r] & \Spec(\C[P_{\varphi}]) \ar[d]\\
    \Spec(\C[\N]) \ar[r]  & \Spec(\C[P]).
    }
\end{align*}
The obvious map $P \to \C[P]$ induces a canonical log structure on $\Spec(\C[P])$, so the above commutative diagram can be upgraded to a commutative diagram of morphisms between log spaces:
\begin{align*}
    \xymatrix
    {
    (\Spec(\C[Q]),Q) \ar[d] \ar[r]  & (\Spec(\C[P_{\varphi}]),P_{\varphi}) \ar[d]\\
    (\Spec(\C[\N]),\N) \ar[r]  & (\Spec(\C[P]),P).
    }
\end{align*}
We denote the family by 
\begin{align*}
\pi: X^{\dag} := \Spec(\C[P_{\varphi}]) \to S^{\dag} := \Spec(\C[P]).
\end{align*}

We arrange the good basis $\varphi_1,\dots,\varphi_{\mu}$ (in Theorem \ref{good basis}) so that $\varphi_1 = 1$ (i.e. of weight degree $0$), $\varphi_2,\dots,\varphi_\nu$ are of weight degree $2$, while the remaining $\varphi_j$'s are of higher weight degrees.
Then $s = \mu - \nu + 1$ is the number of basis elements whose weight degree is not $2$.
Denote $R := \C[P], T := \C[t_1,\cdots,t_s]$ and $\RT := R \otimes_\C T$; in the following, $R$ parametrizes the smoothing of the target $X$ while $T$ parametrizes the deformation of the superpotential $f$.
We define a family of log LG models parametrized by $\ST^{\dag}:=\Spec(\RT) $ (here the log structure is given by the natural map $P \rightarrow \RT$) where the superpotential is given by the function $F = f + t_1 + \sum_{i=2}^s t_i \varphi_{\nu+i-1}$.

Let $m := \C[P-0]$ and $I := \langle t_1,\cdots, t_s \rangle$ be the corresponding maximal ideal of $R$ and $T$ respectively, and $\mathcal{I} = m \otimes T + R \otimes I$ be the maximal ideal of $\RT$. Let $R_k = R/m^{k+1}$ and $\RT_k:=\RT/\mathcal{I}^{k+1}$ be the system of $k^{\text{th}}$-order log rings (with the log structure of $\RT_k$ defined by the canonical monoid homomorphism $P \to  \RT_k$) and $\hat R$ and $\hat \RT$ be the corresponding inverse limits. Let $S^\dag_k$ and $\ST^{\dag}_k$ be the corresponding $k^{\text{th}}$-order log schemes and $\hat S^{\dag}$ and $\hat\ST^\dag$ be their formal counterparts. We denote the $k^{\text{th}}$-order family by $X^\dag_k \rightarrow S^{\dag}_k$ and $\XT^\dag_k \rightarrow \ST^{\dag}_k$ over $S^{\dag}_k$ and $\ST^{\dag}_k$ respectively, with their formal counterparts written as $\hat X^{\dag}$ and $\hat \XT^{\dag}$ respectively. We get a log smooth map $$\phi_k :  \XT^\dag_k \to \ST^\dag_k.$$
 
The $k^{\text{th}}$-order sheaf of log differential and log derivation on $\ST_k^\dag$ will be denoted by $\Omega_{\ST_k^\dag}^\bullet$ and $\Theta_{\ST_k^\dag}^\bullet$ respectively. Similarly, the sheaf of relative log differential and log derivation for the family will be denoted by $\Omega_{ \XT^\dag_k/\ST^\dag_k}^\bullet$ and $\Theta_{\XT^\dag_k/ \ST^\dag_k}^\bullet$ respectively. We define
\begin{align*}
   \mathcal{H}_{F,k} & := H^*(\Theta_{ \XT^\dag_k/ \ST^\dag_k}^\bullet [[u]], \{F,-\} + u \partial),\\
   \mathcal{H}^F_k & := H^*(\Omega_{ \XT^\dag_k/ \ST^\dag_k}^\bullet [[u]], \td F \wedge + u \partial),
\end{align*}
which can be identified by contracting with the relative volume form $\Omega$.

Given $F$ as above, we define
\begin{equation} \label{degree of T}
    \deg(t_i) := 2-\deg{\varphi_i},\quad \forall i
\end{equation}
so that $F$ is homogeneous of weight degree $2$.
To define the weight degree of $z^p \in \C[P]$ for $p \in P$, let us fix a $\Sigma$-piecewise linear function $q_c$ on $|\Sigma|$ such that $q_c(\rho_i) = 2$ for any $1 \leq i \leq d$. Then it defines
\begin{align*}
    q_c: \overline{NE}(Y) \to \Z, \quad p \mapsto \abracket{q_c,p}.
\end{align*}
We define
\begin{equation} \label{degree of P}
    \deg(z^p) := \abracket{q_c,p}.
\end{equation}
It is not difficult to see that this definition of $\deg(z^p)$ is compatible with $\deg(z_i) = 2$ for any $1 \leq i \leq d$ in the sense that if $z^p = \prod_{i=1}^d z_i^{a_i}$, then $\deg(z^p) = 2\sum_i a_i$.

Given (\ref{degree of P}), we get a logarithmic vector field $E_P$ on $\Spec \,\C[P]$. Together with (\ref{degree of T}), we get an Euler vector field $E$ on $\Spec(\C[P_{\varphi}]) \times \C^s$ defined by
\begin{equation*}
    E = E_P + \sum_{i=1}^s \deg(t_i) t_i \frac{\pa}{\pa t_i}.
\end{equation*}
Define
\begin{equation*}
    E_F := E_f + E.
\end{equation*}
Then $E_F F = 2F$ and $[E_F,Q_F] = 0$. In particular, $E_F$ induces a $\Z$-grading on the hypercohomology $\mathcal{H}_{F,k} := H^*(\Theta_{\XT^\dag_k/ \ST^\dag_k}^\bullet [[u]], \{F,-\} + u \partial)$, which is compatible for different values of $k$.

\begin{theorem}[Freeness of Hodge bundle] \label{freeness of Hodge bundle}
    The hypercohomology
    $$\mathcal{H}^F_k := H(\Omega_{ \XT^\dag_k/ \ST^\dag_k}^\bullet [[u]], \td F \wedge + u \partial)$$
    is a graded free $\RT_k[[u]]$-module for each $k$.
\end{theorem}
\begin{proof}
We begin with the fact that $H(\Omega_{ \XT^\dag_k/ \ST^\dag_k}^\bullet , \td F \wedge)$ is a free $\RT_k$-module for each $k$. Notice that the map
$$
 H^*(\Omega_{\XT^\dag_k/ \ST^\dag_k}^\bullet, \td F \wedge) \rightarrow  H^*(\Omega_{X^{\dag}/ 0^\dag}^\bullet , \td f \wedge)
$$
induced by taking quotient $\RT/\mathcal{I}^{k+1} \rightarrow \RT/\mathcal{I} \cong \mathbb{C}$ is always surjective. This is because the product of monomials $\varphi_i = P(\tau_i)$'s with the relative volume form $\Omega$ lift to $H^*(\Omega_{\XT^\dag_k/ \ST^\dag_k}^\bullet, \td F \wedge) $. This further shows that $H^*(\Omega_{\XT^\dag_k/ \ST^\dag_k}^\bullet, \td F \wedge) $ is free $\RT_k$-module for each $k$, with a basis $\varphi_i$'s. 
By the same degree reason and spectral sequence argument as in Theorem \ref{Hodge-to-de Rham}, we conclude that
$$H(\Omega_{ \XT^\dag_k/ \ST^\dag_k}^\bullet[[u]] , \td F \wedge + u\partial)$$
is a free $\RT_k[[u]]$-module.
\end{proof}

\subsection{Gauss-Manin connection and the higher residue pairing}\label{subsec:GM_connection_and_higher_residue}

Define a filtration $\mathfrak{F}_{\geq r}$ on $\Omega_{\hat \XT^\dag}^\bullet$ by the image of the map $\hat \phi^*(\Omega_{\hat \ST^\dag}^r) \otimes \Omega_{\hat \XT^\dag}^\bullet \rightarrow \Omega_{\hat \XT^\dag}^\bullet$ given by taking wedge product. Therefore we have
\begin{align*}
    \mathfrak{F}_{\geq 0}\Omega_{\hat \XT^\dag}^\bullet/\mathfrak{F}_{\geq 1}\Omega_{\hat \XT^\dag}^\bullet \cong & \Omega_{\hat \XT^\dag/\hat \ST^\dag}^{\bullet},\\
    \mathfrak{F}_{\geq 1}\Omega_{\hat \XT^\dag}^\bullet/\mathfrak{F}_{\geq 2}\Omega_{\hat \XT^\dag}^\bullet \cong & \hat \phi^*(\Omega_{\hat \ST^\dag}^1) \otimes \Omega_{\hat \XT^\dag/\hat \ST^\dag}^{\bullet-1}.
\end{align*}
Consider the exact sequence of complexes 
$$0 \to \mathfrak{F}_{\geq 1}\Omega_{\hat \XT^\dag}^\bullet/\mathfrak{F}_{\geq 2}\Omega_{\hat \XT^\dag}^\bullet \to\mathfrak{F}_{\geq 0}\Omega_{\hat \XT^\dag}^\bullet/\mathfrak{F}_{\geq 2}\Omega_{\hat \XT^\dag}^\bullet \to \mathfrak{F}_{\geq 0}\Omega_{\hat \XT^\dag}^\bullet/\mathfrak{F}_{\geq 1}\Omega_{\hat \XT^\dag}^\bullet \to 0,$$
where all the three complexes have differential $\td F \wedge + u \pa$.

The connecting morphism on hypercohomology induces a flat connection
\begin{equation*}
    \nabla: \mathcal{H}^F \to \frac{1}{u} \Omega_{\hat \ST^\dag}^1 \otimes \mathcal{H}^F.
\end{equation*}
Explicitly, for $s \in \mathcal{H}^F$ and $v \in \Theta_{\hat \ST^\dag}^1$, we have
\begin{equation}\label{eqn:GM_connection_formula}
    \nabla_{v} s = \frac{1}{u}\iota_v (\td F \wedge + u \partial) \, \tilde{s},
\end{equation}
where $\tilde{s}$ is a lifting of a representative of $s$ to $\mathfrak{F}_{\geq 0}\Omega_{\hat \XT^\dag}^\bullet/\mathfrak{F}_{\geq 2}\Omega_{\hat \XT^\dag}^\bullet$. Now $\Omega$ can be viewed as a global nowhere vanishing element in $\Omega_{\hat \XT^\dag/\hat \ST^\dag}^n$.
By conjugating with $\Omega$, $\nabla$ induces a flat connection, which will still be denoted by $\nabla$, on $\mathcal{H}_F$ satisfying
\begin{equation*}
    (\nabla_{v} \varphi) \vdash \Omega = \frac{1}{u}\iota_v (\td F \wedge + u \partial) \, (\varphi \vdash \Omega)
\end{equation*}
for $\varphi \in \mathcal{H}_F$ and $v \in \Theta_{\hat \ST^\dag}^1$.

The connection $\nabla$ can be extended to the $u$-direction by defining
\begin{equation*}
    \nabla_{u \frac{\pa}{\pa u}} s := \left(u \frac{\pa}{\pa u} + p - \frac{F}{u}\right)s, 
    \text{ for }s \in \Theta_{\hat \XT^\dag/\hat \ST^\dag}^p [[u]].
\end{equation*}
It can be verified that $\nabla_{u \frac{\pa}{\pa u}}$ is well-defined on $\mathcal{H}_F$ and $[\nabla_v, \nabla_{u \frac{\pa}{\pa u}}]=0$. This extended connection is called the \emph{extended Gauss-Manin connection}. We can equally describe the extended Gauss-Manin connection on $\mathcal{H}_{F,k}$ for each $k$ using the same formulation. 

For any $v\in \Theta^1_{\hat \ST^\dag}$, the Gauss-Manin connection $\nabla$ induces a $\C((u))$-linear map $\nabla_v^{(0)}$ on $\mathcal{H}_{f,\pm}$, which is called the \emph{residue} of $\nabla$. It exists because the Gauss-Manin connection $\nabla$ has log poles and hence nontrivial monodromy groups.
Explicitly,
\begin{align*}
	(\nabla^{(0)} \varphi) \vdash \Omega = \frac{1}{u} (\td f + u \pa) \left(\varphi \vdash \Omega \right).
\end{align*}
Let $\varphi_1,\cdots,\varphi_\mu$ be the good basis in Theorem \ref{good basis}.

\begin{lemma}\label{residue of Gauss-Manin}
For any $v\in \Theta^1_{\hat \ST^\dag}$, there is a nilpotent matrix $N_v \in \C^{\mu\times \mu}$ such that
\begin{equation*}
    u\nabla_v^{(0)} (\varphi_1,\cdots,\varphi_\mu) = (\varphi_1,\cdots,\varphi_\mu) N_v.
\end{equation*}
\end{lemma}
\begin{proof}
    Assume for simplicity that $\varphi = z_1\cdots z_k$, $k\leq n$ is an element of the good basis. Then by definition,
    \begin{align*}
        (u\nabla^{(0)} \varphi) \vdash \frac{\td z_1}{z_1} \wedge \cdots \wedge \frac{\td z_n}{z_n}
        =& (\td f + u \pa) \left(z_1 \cdots z_k \frac{\td z_1}{z_1} \wedge \cdots \wedge \frac{\td z_n}{z_n}\right) \\
        =& \sum_{l=n+1}^d z_l z_1 \cdots z_k \frac{\td z_l}{z_l} \wedge \frac{\td z_1}{z_1} \wedge \cdots \wedge \frac{\td z_n}{z_n}.
    \end{align*}
    Using the equation $z_l = \prod_{i=1}^n z_i^{a_{li}} z^{p_l}$ where $p_l \in P-0$, we have $\frac{\td z_l}{z_l} = \frac{\td z^{p_l}}{z^{p_l}} + \sum_{i=1}^n a_{li} \frac{\td z_i}{z_i}$. Hence
    \begin{equation*}
        u\nabla^{(0)}_v \varphi = \sum_{l=n+1}^d \frac{\td z^{p_l}}{z^{p_l}}(v) \cdot z_l z_1 \cdots z_k,
    \end{equation*}
    and it suffices to show $z_l z_1 \cdots z_k$ can be written as a $\C$-linear combination of the good basis.
    This can be done by an argument parallel to that of page 107 in \cite{Ful}, using the Moving Lemma \ref{algebraic moving lemma} and descending induction on the subscript $i$ of $\varphi_i$.
    Thus $N_v \in \C^{\mu\times \mu}$ for any $v\in \Theta^1_{\hat \ST^\dag}$.
    To see that $N_v$ is nilpotent, we only need to notice that $u\nabla_v^{(0)}$ always increases the weight degree by $2$.
\end{proof}

Now we define a trace map for sections in $\mathcal{H}_F$, which will turn out to be a family version of that for $\mathcal{H}_f$.
Since $F = f + t_1 + \sum_{i=2}^s t_i \varphi_{\nu+i-1}$, we can write 
$$\td F = \sum_{i=1}^d z_i\frac{\pa F}{\pa z_i} \frac{\td z_i}{z_i},$$
where $z_i\frac{\pa F}{\pa z_i}$ is a polynomial in $z_1,\cdots,z_d$ and $t$.
By the change of frame relations (\ref{relation of log differential}) that still hold in $\Omega_{\hat \XT^\dag/\hat \ST^\dag}^1$, we have
\begin{equation}\label{eqn:log_differential_change_of_frame}
\left(\frac{\td z_{1}}{z_{1}}, \cdots, \frac{\td z_{n}}{z_{n}}\right) C_i = \left(\frac{\td z_{i_1}}{z_{i_1}}, \cdots, \frac{\td z_{i_n}}{z_{i_n}}\right) 
\end{equation}
as in the proof of Lemma \ref{V_f identity}. Therefore $\td F$ can be written as
\begin{equation*}
    \td F = \sum_{k=1}^n h_k \frac{\td z_k}{z_k},
\end{equation*}
where $h_k$ is again a polynomial of $z_1,\cdots,z_d$ and $t$. 

Let $X_i$ be an $n$-dimensional component of $\hat X_T$ with the corresponding cone generated by the rays $\rho_{i_1},\cdots,\rho_{i_n}$. We express $h_k|_{X_i}$ as a meromorphic function of $z_{i_1},\cdots,z_{i_n}$, $z^{p_i}$'s and $t_i$'s using $z_l = \prod_{i=1}^n z_i^{a_{li}} z^{p_l}$. We get a meromorphic differential form on $X_i \times \hat \ST$, which we denote by $(\td F)_i$.
Clearly, $(\td F)_i$ and $(\td F)_j$ can be transformed to each other using equation \eqref{eqn:log_differential_change_of_frame} and the defining equation of the map $\Spec\, \C[P_{\varphi}] \to \Spec\, \C[P]$, follows from the fact that $\td F$ is a global log differential in $\Omega_{\hat \XT^\dag/\hat \ST^\dag}^1$.
When $z^{p_i}=t_i=0$, $\td F$ reduces to $\td f$.

Let $h_k$ be as above and $\bar g_k$ be defined as in (\ref{V_f}), we formally write
\begin{equation*}
    V_F := \frac{1}{\sum_{k=1}^n h_k \bar g_k} \left(\sum_{k=1}^n \bar g_k \theta_k + \sum_{l=n+1}^d \left(\sum_{k=1}^n a_{kl} \bar g_k\right) \theta_l \right) = \frac{1}{\sum_{k=1}^n h_k \bar g_k} \sum_{k=1}^n \bar g_k \tilde{\theta}_k.
\end{equation*}
It is only a formal expression but we can assign a meromorphic derivation on each component $X_i \times \hat \ST$ of $\hat \XT$ as follows.
Let $X_i$ be as before, for any $j$ such that $j \neq i_s,\, \forall 1 \leq s \leq n$, we treat $\bar z_j|_{X_i} = 0$ in the defining expression of $\bar g_k$ so that $\bar g_k|_{X_i}$ is a linear function of $\bar z_{i_1}, \cdots,\bar z_{i_n}$. In particular, the only common zeros of $\bar g_k|_{X_i}$ is the origin.
By writing $h_k|_{X_i}$ as a meromorphic function and making use of the change of frame
$$
(\tilde{\theta}_{i_1},\cdots , \tilde{\theta}_{i_n}) C_i^T  = (\tilde{\theta}_1,\cdots, \tilde{\theta}_n),
$$
we define $(V_F)_i$ to be a section of $\A^{0,0}_{X_i} \otimes_{\cO_{X_i}} (\Theta_{\hat \XT^\dag/\hat \ST^\dag}^1)|_{X_i}(\star D_i)$.
Since $z^{p_i}$'s and $t_i$'s are formal variables, $(V_F)_i$ can be written as a meromorphic derivation on $X_i-0$ valued in $\hat \RT$.
When $z^{p_i}=t_i=0$, $(V_F)_i$ reduces and glues to $V_f$.
However, due to the presence of $\bar z_i$'s variables, the above expression for $V_F$ does not define a global derivation on $\hat \XT$.
Nevertheless, we have the following result that imitates Lemma \ref{V_f identity}.

\begin{lemma}
For any $i$, let $\{F,-\}_i$ be defined using $(\td F)_i \wedge$ by conjugating with $\Omega$. Then
\begin{equation*}
    [\{F,-\}_i, (V_F)_i] = 1 \quad \text{on  } X_i-0,
\end{equation*}
as an action on sections of $\A^{0,*}_{X_i} \otimes_{\cO_{X_i}} (\Theta_{\hat \XT^\dag/\hat \ST^\dag}^*)|_{X_i}(\star D_i)$. 
\end{lemma}
\begin{proof}
    With the above assumption on notation, the proof is parallel to that of Lemma \ref{V_f identity}.
\end{proof}

Let $T_{\rho,i}^u$ and $R_{\rho,i}^u$ be defined as in the proof of Proposition \ref{homotopy formula2}, with $V_f$ replaced by $(V_F)_i$. Then we have
\begin{align*}
    [Q_F,R_{\rho,i}^u] = 1 - T_{\rho,i}^u \qquad \text{on} \quad \Gamma(X_i-0, \A^{0,q}_{X_i} \otimes_{\cO_{X_i}} \Omega^p_{\hat \XT^\dag/\hat \ST^\dag}|_{X_i}(\star D_i)).
\end{align*}
For $\alpha \in \Theta^\bullet_{\hat \XT^\dag/\hat \ST^\dag}$, we can define a trace as
\begin{equation*}
    \Tr(\alpha) := \sum_i \dashint_{X_i} (T_\rho^u \alpha \vdash \Omega) \wedge \Omega \in \hat \RT[[u]],
\end{equation*}
where $\dashint$ denotes the \emph{regularized integral} in Appendix \ref{Section of regularized integral} (see Definition \ref{regularized integral}).
To see why this regularized version of integral is needed, one only needs to note that the integrand may have arbitrary orders of poles.

\begin{proposition} \label{vanishing of trace 2}
    The trace map satisfies $\Tr((\{F, -\} + u \pa) \beta) = 0, \, \forall \beta \in \Theta^\bullet_{\hat \XT^\dag/\hat \ST^\dag}$. Hence the trace map descends to $\mathcal{H}_F$.
\end{proposition}
\begin{proof}
    Since $T^u_{\rho,i} Q_F = Q_F T^u_{\rho,i}$ for $Q_F  = dF\wedge + \bar\partial + u\partial$, we need only to show that for any $\varphi = \beta \vdash \Omega \in \Omega^{n-1}_{\hat \XT^\dag/\hat \ST^\dag}$,
    \begin{align*}
        \sum_i \dashint_{X_i} Q_F T^u_{\rho,i} \varphi \wedge \Omega = 0,
    \end{align*}
    which is equivalent to
    \begin{align*}
        \sum_i \dashint_{X_i} d (T^u_{\rho,i} \varphi \wedge \Omega) = 0.
    \end{align*}
    By Theorem \ref{Stokes formula 2}, it is sufficient to show that the higher residue satisfies
    \begin{equation*}
        \Res_{D_{ij}} (T^u_{\rho,i} \varphi \wedge \Omega)_i + \Res_{D_{ij}} (T^u_{\rho,j} \varphi \wedge \Omega)_j = 0
    \end{equation*}
    for any pair of $n$-dimensional components of $X_i, X_j$ such that $D_{ij} = X_i \cap X_j$ has dimension $n-1$; here the subscripts $i$ and $j$ indicate restriction to the corresponding component of $\hat \XT$.

    Assume that $X_i$ has free coordinates $(x, w_1,\cdots,w_{n-1})$ and $X_j$ has free coordinates $(y, w_1,\cdots,w_{n-1})$.
    As $\Sigma$ is smooth, $\hat \XT$ is contained in $\{xy = r\}$ for some function $r = r(w_1,\cdots,w_{n-1}, z^p)$ which can be meromorphic in the $w_1,\cdots, w_{n-1}$ variables.
    By equation \eqref{vanishing residue}, it suffices to show that
    \begin{equation} \label{residue cancel}
        \Res_{D_{ij}} ((T^u_{\rho,i} \varphi \wedge \Omega)_i) + \Res_{D_{ij}} ((T^u_{\rho,j} \varphi \wedge \Omega)_j) = 0.
    \end{equation}
To compute the residues, we express $(T^u_{\rho,i} \varphi \wedge \Omega)_i = \sum_{p,l} \phi_{p,l} z^{p} t^l$ (here $l$'s are multi-indices) and every $\phi_{p,l}$ as
$$
\phi_{p,l} = d \bar x \wedge \alpha_{p,l} + \frac{d\log x}{x^{k-1}} \wedge \beta_{p,l} + \gamma_{p,l}
$$
as in equation \eqref{eqn:higher_residue_definition_decomposition}. The residue is obtained from applying $(\frac{\pa}{\pa x})^{k-1}$ to $\beta_{p,l}$ and then restricting on $D_{ij}$. It is similar for $(T^u_{\rho,j} \varphi \wedge \Omega)_j = \sum_{p,l} \psi_{p,l} z^{p} t^l$

    According to the expression of $T^u_{\rho,i}$ in the proof of Proposition \ref{homotopy formula2}, we need to consider $\rho$ and $Q(\rho) = (\bar \pa + u \pa) \rho$. By definition of $\rho$ we have $\rho_i |_{D_{ij}} = \rho_j |_{D_{ij}}$ and $(\frac{\pa}{\pa x})^k (\rho_i)|_{D_{ij}} = 0 = (\frac{\pa}{\pa y})^k (\rho_j)|_{D_{ij}}$ for any $k>0$. Writing $Q(\rho_i) = d\bar x \wedge \mathfrak{a}_i + dx \wedge \mathfrak{b}_i+\mathfrak{c}_i$ and $Q(\rho_j) = d\bar y \wedge \mathfrak{a}_j + dy \wedge \mathfrak{b}_j+\mathfrak{c}_j$, we see that $(\frac{\pa}{\pa x})^k (\mathfrak{b}_i)|_{D_{ij}} = 0 = (\frac{\pa}{\pa y})^k (\mathfrak{b}_j)|_{D_{ij}}$ for any $k\geq 0$, $\mathfrak{c}_i |_{D_{ij}} = \mathfrak{c}_j |_{D_{ij}}$ and $(\frac{\pa}{\pa x})^k (\mathfrak{c}_i)|_{D_{ij}} = 0 = (\frac{\pa}{\pa y})^k (\mathfrak{c}_j)|_{D_{ij}}$ for any $k>0$. 
    For any $1 \leq k \leq n$, we have $(\bar g_k)_i|_{D_{ij}} = (\bar g_k)_j|_{D_{ij}}$ and $(\frac{\pa}{\pa x})^k((\bar g_k)_i) = 0 = (\frac{\pa}{\pa y})^k ((\bar g_k)_j)$ for any $k>0$. From these, we can replace the terms involving $\rho$, $\mathfrak{b}_i$, $\mathfrak{c}_i$ and $(\bar g_k)_i$ as constant functions along the $x$ variable in the computation of $\Res_{D_{ij}} ((T^u_{\rho,i} \varphi \wedge \Omega)_i)$. Their restrictions to $D_{ij}$ agree with $\rho$, $\mathfrak{b}_j$, $\mathfrak{c}_j$ and $(\bar g_k)_j$.

    If we further substitute $y$ by $\frac{r}{x}$ on $X_j$ and change $d\log y$ into $d \log(\frac{r}{x})$, then $(T^u_{\rho,j} \varphi \wedge \Omega)_j$ transforms into $(T^u_{\rho,i} \varphi \wedge \Omega)_i$.
    From the expression of $T^u_{\rho,i}$, we can expand $\phi_{p,l}$ into a Laurent series of $x$:
    $$
    \phi_{p,l} = \sum_{r\in \mathbb{Z}} \phi_{p,l,r} x^r,
    $$ 
    where the coefficients $\phi_{p,l,r}$'s are smooth functions on $D_{ij}$ that are constant along $x$. Then $\Res_{D_{ij}} (T^u_{\rho,i} \varphi \wedge \Omega)_i = \sum_{p,l} \Res_{D_{ij}}(\phi_{p,l}) z^p t^l$, where $\Res_{D_{ij}}(\phi_{p,l})$ equals the coefficient of $\frac{\td x}{x}$ in the above Laurent series expansion of $\phi_{p,l}$.
    By the above argument, this coefficient is the negative of the $y$-independent coefficient of $\frac{\td y}{y}$ of the Laurent expansion of $\psi_{p,l}$ in $y$.
    Therefore (\ref{residue cancel}) follows.
\end{proof}
With the trace map at hand, we define a family version of higher residue pairing: for any $\alpha,\beta \in \mathcal{H}_F$, we set
\begin{equation*}
    \KK_F(\alpha, \beta) = Tr(T_\rho^u \alpha \cdot \overline{T_\rho^u \beta})
\end{equation*}
By Proposition \ref{vanishing of trace 2}, $\KK_F$ descends to a pairing:
\begin{equation*}
    \KK_F: \mathcal{H}_F \times \mathcal{H}_F \longrightarrow \hat \RT [[u]].
\end{equation*}

As $[Q_F,E_F] = 0$, there exists a weight grading on $\mathcal{H}_F$. It turns out that $\KK_F$ is homogeneous with respect to this $\Z$-grading.
\begin{lemma}
The higher residue pairing $\KK_F$ is homogeneous of weight degree $-2n$.
\end{lemma}
\begin{proof}
    Similar to that of Lemma \ref{weight property}.
\end{proof}

\begin{theorem} \label{properties of HRP}
Let $s_1, s_2$ be sections of $\mathcal{H}_F$. Then
\begin{enumerate}[(1)]
    \item $\KK_F(s_1,s_2) = \overline{\KK_F(s_2,s_1)}$;
    \item $\KK_F(g(u)s_1,s_2) = \KK_F(s_1,g(-u)s_2) = g(u) \KK_F(s_1,s_2)$ for any $g(u) \in \hat \RT [[u]]$;
    \item $v \KK_F(s_1,s_2) = \KK_F(\nabla_v s_1,s_2) + \KK_F(s_1, \nabla_v s_2)$ for any $v \in \Theta^1_{\hat \ST^\dag}$;
    \item $(u\frac{\td}{\td u} + n)\KK_F(s_1,s_2) = \KK_F(\nabla_{u\frac{\td}{\td u}} s_1,s_2) + \KK_F(s_1, \nabla_{u\frac{\td}{\td u}} s_2)$;
    \item the pairing $\mathcal{H}_F / u\mathcal{H}_F \times \mathcal{H}_F / u\mathcal{H}_F \longrightarrow \hat \RT$ induced by $\KK_F$ is nondegenerate.
\end{enumerate}
\end{theorem}
\begin{proof}
    Statements (1) and (2) are trivial. Statement (5) follows from Proposition \ref{nondegeneracy}. Statements (3) and (4) follow from direct calculations, similar to that of \cite[Proposition 3.20]{LLS}.
\end{proof}
By Theorems \ref{freeness of Hodge bundle} and \ref{properties of HRP}, the triple $(\mathcal{H}_F, \nabla, \KK_F)$ gives a \emph{log semi-infinite variation of Hodge structure} (abbrev. $\frac{\infty}{2}$-LVHS).

\subsection{Logarithmic Frobenius manifold}\label{log Frobenius}

In this section, we show that there is a logarithmic Frobenius manifold structure on $\ST$. By \cite{R} and \cite{KLM}, it suffices to show the existence of a good opposite filtration and a primitive form.
Our exposition is along the lines of \cite{LLS, LLSS}; some details will be omitted and we refer the reader to those references.

We start with the construction of a \emph{good opposite filtration} \cite{LLS}.
Due to the nontriviality of $\nabla_v^{(0)}$, there exist no global flat sections of 
$\mathcal{H}_{F,\pm} = H^*(\Theta_{\hat \XT^\dag/\hat \ST^\dag}^\bullet [[u]][u^{-1}], \{F,-\} + u \partial)$. However, there is the following Deligne extension of sections in $\mathcal{H}_{f,\pm} := H(\PV_{\log}(X)[[u]][u^{-1}], \{f,-\} + u \partial)$.

\begin{lemma} \label{Deligne extension}
There is a degree-preserving $\C((u))$-linear map $DE: \mathcal{H}_{f,\pm} \to \mathcal{H}_{F,\pm}$ such that
\begin{equation*}
    DE(\varphi) \equiv \, \varphi \, (\text{mod } \mathcal{I}), \quad \nabla_v \circ DE = DE \circ \nabla_v^{(0)}
\end{equation*}
for any $ v \in \Theta_{\hat \ST^\dag}^1$.
\end{lemma}
\begin{proof}
Since $\varphi_1,\cdots,\varphi_\mu$ can also be viewed as a basis of sections of $\mathcal{H}_F$, there exists an $\hat \RT((u))$-valued matrix $A$ such that
\begin{equation} \label{extension matrix}
    (DE (\varphi_1),\cdots,DE (\varphi_\mu))= (\varphi_1,\cdots,\varphi_\mu) A.
\end{equation}
Applying $u\nabla$ on both sides and let $N_p$ and $N$ be the matrix of $\nabla$ and $\nabla^{(0)}$ respectively, we have
\begin{align*}
    (\varphi_1,\cdots,\varphi_\mu) N_p A + (\varphi_1,\cdots,\varphi_\mu) u \nabla A = (\varphi_1,\cdots,\varphi_\mu) AN
\end{align*}
and hence $u\nabla A = AN-N_p A$.
We solve for $A$ by induction on $k$ applying to the equation
\begin{equation} \label{equation of Deligne extension}
    u\nabla A \equiv AN-N_p A \, (\text{mod } \mathcal{I}^{k+1}).
\end{equation}
The $k=0$ case is trivial.
The induction part is the same as that of \cite[Lemma 6.13]{KLM}.
\end{proof}

\begin{lemma} \label{HRP on extension of good basis}
Let $DE (\varphi_1),\cdots,DE (\varphi_\mu)$ be the Deligne extension of the good basis in Theorem \ref{good basis}. Then
\begin{equation*}
    \KK_F(DE(\varphi_i), DE(\varphi_j)) = \KK_f(\varphi_i,\varphi_j) \in \C.
\end{equation*}
\end{lemma}
\begin{proof}
    This follows from (3) of Theorem \ref{properties of HRP}, Lemma \ref{Deligne extension}, $(\nabla^{(0)})^2 = 0$ and Lemma \ref{residue of Gauss-Manin} which says that $\nabla^{(0)}_v$ is nilpotent for any $v \in \Theta^1_{\hat \ST^\dag}$.
\end{proof}
\begin{remark}
    The sections $DE(\varphi), \varphi \in \mathcal{H}_{f,\pm}$ play the role of flat sections in the usual Landau-Ginzburg model.
\end{remark}

\begin{proposition}\label{good opposition}
Let $\mathcal{H}_{F,-}$ be the $\hat \RT[u^{-1}]$-submodule of $\mathcal{H}_{F,\pm}$ generated by $u^{-1}DE(\varphi_1)$, $\cdots$, $u^{-1}DE(\varphi_\mu)$. Then $\mathcal{H}_{F,-}$ defines a good opposite filtration, in other words, we have
\begin{enumerate}[(a)]
    \item $\mathcal{H}_{F,\pm} = \mathcal{H}_F \oplus \mathcal{H}_{F,-}$;
    \item $\mathcal{H}_{F,-}$ is preserved by $\nabla_v$ for any $v \in \Theta^1_{\hat \ST^\dag}$;
    \item $\mathcal{H}_{F,-}$ is preserved by $E_F$;
    \item $\mathcal{H}_{F,-}$ is isotropic with respect to the symplectic pairing $\Res_{u=0}\KK_F(-,-)$.
\end{enumerate}
\end{proposition}
\begin{proof}
    Let $\mathcal{H}_{f,-}$ be the $\C[u^{-1}]$-module generated by $u^{-1}\varphi_1,\cdots,u^{-1}\varphi_\mu$. Then $\mathcal{H}_{f,\pm} = \mathcal{H}_f \oplus \mathcal{H}_{f,-}$ by Theorem \ref{Hodge-to-de Rham}. Thus 
    (a) follows from $DE(\varphi_i) \equiv \, \varphi_i \ (\text{mod } \mathcal{I})$.
    (b) follows from $\nabla_v \circ DE = DE \circ \nabla_v^{(0)}$ and Lemma \ref{residue of Gauss-Manin}.
    (c) follows from Lemma \ref{Deligne extension} that $u^{-1}DE(\varphi_1),\cdots,u^{-1}DE(\varphi_\mu)$ are all homogeneous with respect to the $\Z$-grading induced by $E_F$.
    (d) follows from Lemma \ref{HRP on extension of good basis}.
\end{proof}

To construct a \emph{primitive form}, we consider a \emph{Birkhoff factorization} of the $\hat \RT((u))$-valued matrix $A$ in the proof of Lemma \ref{Deligne extension}. Namely, we solve the equation $A \equiv CB^{-1} \ (\text{mod } \mathcal{I}^{k+1})$ by induction on $k$, so that $C$ is an $\hat \RT[[u]]$-valued matrix and $B$ is of the form $\text{Id}_{\mu\times \mu} + B'$, in which $B'$ is an $u^{-1}\hat \RT[u^{-1}]$-valued matrix (cf. \cite[Sections 5.4 and 6]{LLS}).
Then we have
\begin{equation} \label{spliting}
    (DE (\varphi_1),\cdots,DE (\varphi_\mu)) B = (\varphi_1,\cdots,\varphi_\mu) C.
\end{equation}
\begin{lemma} \label{lemma_primitive_form}
Denote $(\tilde\varphi_1, \cdots, \tilde\varphi_\mu) :=  (DE (\varphi_1),\cdots,DE (\varphi_\mu)) B$. Then $\KK_F(\tilde\varphi_i,\tilde\varphi_j) \in \hat \RT$.
\end{lemma}
\begin{proof}
    Since $C$ is $\hat \RT[[u]]$-valued and $\KK_F(\varphi_i,\varphi_j) \in \hat \RT[[u]]$, we see that $\KK_F(\tilde\varphi_i,\tilde\varphi_j) \in \hat \RT[[u]]$. On the other hand, by Lemma \ref{HRP on extension of good basis} and the form of $B$, we have $\KK_F(\tilde\varphi_i,\tilde\varphi_j) \in \hat \RT[u^{-1}]$. The result follows.
\end{proof}

\begin{proposition}\label{prop:primitive_form}
    Let $\zeta := \tilde\varphi_1$. Then $\zeta$ is a primitive form in the sense of \cite{LLS, KLM}. More concretely,
    \begin{enumerate}[(a)]
        \item $\zeta \in \mathcal{H}_F \cap u\mathcal{H}_{F,-}$;
        \item $\nabla_v \zeta = 0 \in u\mathcal{H}_{F,-}/\mathcal{H}_{F,-}$ for any $v \in \Theta^1_{\hat \ST^\dag}$;
        \item $\zeta$ is homogeneous with respect to the $\Z$-grading induced by $E_F$;
        \item The Kodaira-Spencer map $KS: \Theta^1_{\hat \ST^\dag} \to \mathcal{H}_F/u\mathcal{H}_F$ given by $KS(v) := \nabla_v \zeta \ (\text{mod } u\mathcal{H}_F)$ is a bundle isomorphism.
    \end{enumerate}
\end{proposition}
\begin{proof}
    (a) follows from the definition of $\zeta$. (b) follows from the assumption on $B$ in the equation \eqref{spliting}. (c) is obvious. (d) follows from an easy explicit computation.
\end{proof}

By \cite[Proposition 1.11]{R}, a good opposite filtration together with a primitive form gives the germ of a \emph{logarithmic Frobenius manifold}. Thus we arrive at our main result.

\begin{theorem}\label{main theorem}
Let $(X^\dag, \phi, f)$ be the logarithmic Landau-Ginzburg model defined from a projective toric manifold. Then there is a logarithmic Frobenius manifold structure on the base space $\ST$.
\end{theorem}

\subsection{Log LG mirrors of semi-Fano toric manifolds}\label{semi-Fano}
In this subsection, we investigate the canonical coordinates and primitive forms in the case of \emph{semi-Fano} toric manifolds, meaning projective toric manifolds $X$ whose anticanonical line bundle $K_X^{-1}$ is nef.

To begin with, we arrange the good basis $\varphi_1,\dots,\varphi_{\mu}$ (in Theorem \ref{good basis}) so that $\varphi_1 = 1$ (i.e. of weight degree $0$), $\varphi_2,\dots,\varphi_\nu$ are of weight degree $2$, while the remaining $\varphi_j$'s are of higher weight degrees. Let $\mathbf{F}$ be the vector space spanned by the basis $$DE(\varphi_1),DE(\varphi_2),\dots,DE(\varphi_\nu),DE(\varphi_{\nu+1})\dots,DE(\varphi_\mu),$$
and write its coordinate functions as 
$\tau_1,\log(\tau_2),\dots,\log(\tau_\nu),\tau_{\nu+1}\dots,\tau_\mu$, which will be the flat coordinates.  

We first describe the \emph{semi-infinite period map} (or \emph{period map} for short). We take the vector space $V=P^{gp} \times_{\Z} \C$ and treat an element $p \in P^{gp} \subset V$ as $\log z^{p}$. For $v\in \Theta_{\hat S^{\dag}} \cong \hat R \otimes_{\C} V^*$, we declare its action on $\hat{R}_V = \hat R \hat \otimes_{\C} \text{Sym}^* V $ as a derivation using the natural pairing $\langle\cdot,\cdot\rangle$ between $V^*$ and $V$. Explicitly, if we write $\partial_n$ for $n \in V^*$, then we have $\partial_n (\log z^p) = \langle n,p\rangle$ and $\partial_n(z^p) = \langle n,p\rangle z^p$. Formally, we will extend the coefficient ring from $\RT$ to $\RT\hat \otimes_{\C} \text{Sym}^* V$ in order to write down the flat sections. We let $\fvol = DE(\varphi_1) + \eta$ with $\eta \in \mathcal{H}_{F,\pm} \hat \otimes_{\C} \text{Sym}^{>0} V$ such that $\nabla_v \fvol = 0$ $\forall v \in \Theta_{\hat \ST}$. 

\begin{definition}
	The \emph{semi-infinite period map} is the map
	$$\Psi:\hat \RT\hat \otimes_{\C} \text{Sym}^*(V) \rightarrow \C[\tau_1, \log \tau_2, \dots, \log \tau_\nu, \tau_{\nu + 1}, \dots, \tau_\mu]$$
    given by
	$$
	\Psi(z^p,\log z^p) = u[\zeta - \fvol] \in u\mathcal{H}_{F,-}/\mathcal{H}_{F,-} \cong \mathbf{F}\otimes_{\C} \hat \RT \hat\otimes_{\C} \text{Sym}^*(V),
	$$
	where $\zeta$ is a primitive form as in Proposition \ref{prop:primitive_form}.
	
\end{definition}


\subsubsection{The period map for log LG mirror of a semi-Fano toric manifold}

We call the variables $\log (\tau_2),\dots,\log(\tau_\nu)$ the \emph{small quantum variables}, because they correspond to a basis of $H^2(Y)$; all the other parameters are called \emph{big quantum variables}.

We will study the period map restricted to $\hat S$ for the family $(X^{\dag},F)$ mirror to a semi-Fano toric manifold (resp. Fano toric manifold) $Y$, which refers to those with all $p \in P$ having non-negative (resp. positive) weight degrees. We denote by $P_0$ the submonoid consisting of those $p$'s with weight degree $0$. 

Let us first discuss the log LG mirror of a Fano toric manifold. We write
\begin{equation}\label{eqn:expression_of_1_in_flat_basis}
1 = f(z^{p},u) DE(\varphi_1) +  \sum_{i=2}^{\nu} g_i(z^p,u) DE(\varphi_i) + \sum_{j=\nu+1}^{\mu} h_j(z^p,u) DE(\varphi_j),
\end{equation}
and expand $f = \sum_k f_k u^k$, $g_i = \sum_{k} g_{ik} u^k$ and $h_j = \sum_k h_{jk} u^k$ into Laurent series in $u$. By a weight degree argument, we see that $g_{ik} = 0$ and $h_{jk} = 0$ for $k\geq 0$. Similarly, we have $f_{k} = 0$ for $k>0$ and $f_{0} = 1$. Thus we conclude that $\zeta = 1$ is the primitive form.

By a similar weight argument, we find that $h_{j(-1)} = 0 = g_{i(-1)}$, and $f_{-1} = \sum_{\text{deg}(p) = 2} a_p z^p$ for some constants $a_p$. Modulo the ideal $\mathcal{I}_{> 2}$ generated by those $z^{p}$'s with $p$ having weight degree greater than $2$, we obtain the formula
$$
1 = (1+u^{-1}\sum_{\text{deg}(p) = 2} a_p z^p) DE(\varphi_1) + u^{-2}\sum_{i} g_{i(-2)} DE(\varphi_i).
$$
We need to show that $a_p = 0$ for every $p$. Taking a vector field $n \in \Theta_{S^{\dag}}$ and applying $\nabla_n$ to both sides, we deduce that
$$
\sum_{s=1}^d u^{-1 } z_s (\iota_{n} d\log z_s) = u^{-1} \sum_{\text{deg}(p) = 2}  a_p \nabla_n(z^p) DE(\varphi_1) + u^{-1} \nabla_{n} (DE(\varphi_1)) + u^{-2} (\cdots),
$$
and it is enough to argue that expressing $z_s$'s in terms of the frame $DE(\varphi_i)$'s does not involve components in $DE(\varphi_1)$ modulo $\mathcal{I}_{>0}$. This is achieved by showing the following Lemma \ref{lem:connection_special_form_semi_fano}, which is valid for the log LG mirror $(X^{\dag},F)$ of a semi-Fano toric manifold. 

We take the basis $\{\varphi_1, \dots, \varphi_\mu\}$ for $\mathcal{H}_{f,+}$ as in the beginning of this subsection, lift it to a frame of $\mathcal{H}_{F,+}$ using the same monomial expressions and call them $\psi_i$'s.

\begin{lemma}\label{lem:connection_special_form_semi_fano}
	Let $(X^{\dag},F)$ be the log LG mirror of a semi-Fano toric manifold $Y$. Then the connection matrix $u^{-1} A$ of $\nabla_n$, as a $\Omega_{S^{\dag}}^1[[u]]$-valued matrix and written as a block matrix with respect to the frames $\psi_1$ in the first block, $\psi_2,\dots,\psi_{\nu}$ in the second block and $\psi_{\nu+1},\dots,\psi_{\mu}$ in the third block, is of the form
	$$
	A = \begin{Bmatrix} 0 & 0 & 0 \\
		* & * & * \\
		* & * & * 
		\end{Bmatrix}
	$$
	when modulo the ideal $\mathcal{I}_{>2}$. 
\end{lemma}

\begin{proof}
	For $\psi_1 = 1 $, we have $\nabla_n 1 = \sum_{s=1}^d u^{-1} z_s \iota_n(d\log z_s)$ and we can express each $z_s$ in terms of $\psi_2,\dots,\psi_\nu$ using relations in $\mathcal{H}_{F,\pm}$ as in the Moving Lemma \ref{algebraic moving lemma}. 
	
	For other $\psi_i = P(\tau_i) = z_{i_1}\cdots z_{i_k}$ which is identified with $P(\tau_i) \Omega \in \mathcal{H}^{F}_{+}$, the action of the connection $\nabla_n$ is given by
	$$
	\nabla_n(P(\tau_i)\Omega) = \iota_n \left( P(\tau_i) \left( u \sum_{l=1}^k d\log z_{i_l}  + \sum_{s=1}^d z_s d\log z_s \right)\wedge \Omega \right),
	$$ 
	using the description in equation \eqref{eqn:GM_connection_formula}. 
	Note that the possible relations in $\C[P_{\varphi}]$ are 
	$$
	\prod_{i \in I} z_i = z^{p} \prod_{j\in J} z_{i_j},
	$$
	where $I$ is a subset of $\{1,\dots,d\}$, $p \in P$ and $i_j \in \{1,\dots,d\}$ (which is allowed to repeat). If we further modulo $z^{p} \in \mathcal{I}_{>2}$, we observe that the only possible relations will have $J \neq \emptyset$. Therefore, the RHS cannot have components along $\psi_1 = 1$. 
\end{proof}

In the Fano case, we see that, modulo $\mathcal{I}_{>2}$, we have
$$
\mathbb{1} = DE(\varphi_1) - \sum_{i=1}^{\nu} (\sum_{l=1}^{\nu} c_{il} \log(q_l)) DE(\varphi_i) + u^{-2} (\cdots).
$$
The period map restricted to $\hat R \otimes_\C \text{Sym}^*V$ is simply given by 
\begin{align*}
\log(\tau_i) &= \sum_{il} c_{il} \log(q_l),& \text{for $i = 2,\dots,\nu$,}\\
\tau_j & = 0, &\text{for $j \neq 2,\dots,\nu$,}
\end{align*}
where $\log(q_l)$'s is a basis for $V$. 

In the semi-Fano case, we have a more general expression.
\begin{theorem}\label{small period map}
	Let $(X^{\dag},F)$ be the log LG mirror of a semi-Fano toric manifold. When restricted to the parameter space $\hat R \otimes_\C \text{Sym}^*V$, $\zeta=1$ is a primitive form and the period map takes the form 
	\begin{align*}
		\log(\tau_i) &= \sum_{il} c_{il} \log(q_l) + g_{i(-1)}(z^{p}),&  \text{for $i = 2,\dots,\nu$,}\\
		\tau_j & = 0, &\text{for $j \neq 2,\dots,\nu$,}
	\end{align*}
    where each $g_{i(-1)}(z^{p})$ is a formal series of $z^{p}$ where $p$'s have weight degree $0$. 
\end{theorem}

\begin{proof}
	Recall that in equation \eqref{eqn:expression_of_1_in_flat_basis}, we have $g_{ik} = 0 = h_{jk}$ for $k\geq 0$, $f_k=0$ for $k>0$ and $h_{j(-1)} = 0$. The only possibility that cannot be ruled out by a degree argument is when $f_{0}$ and $g_{i(-1)}$ are series in $z^{p}$ where $p$'s have weight degree $0$. The presence of $g_{i(-1)}$ contributes to the corresponding terms in $\log(\tau_i)$, so we simply have to rule out the presence of $f_0$ and $f_{-1}$. Once again we have to look at the expression of $\psi_1$ in terms of $DE(\varphi_i)$'s.
	
	Fix a decreasing sequence of ideals $\mathcal{J}_i$ containing $\mathcal{I}_{>2}$ such that $\mathcal{J}_1 = m$, $\mathcal{J}_{i}/ \mathcal{J}_{i+1}$ is one-dimensional and generated by $z^{p}$ for some $p$, and their intersection is $\mathcal{I}_{>2}$. Such a sequence can be found because $P$ is a toric monoid, using arguments similar to those in \cite[below Lemma 4.17]{KLM}. We define a $\mu \times \mu$ matrix $G$ by the  change of frames $[\psi_1,\dots,\psi_\mu]  = [DE(\varphi_1),\dots,DE(\varphi_\mu)]G$. Passing to the quotient $R((u))/\mathcal{J}_i((u))$, we write the corresponding matrix as $G_i$, which are compatible for different $i$'s. Notice that $G_1 = \text{id}$.  
	
	Write $\Omega_{\hat S^{\dag}}^1 \cong \hat R \otimes_{\C} V$ with elements in $V$ written as $d\log q_i$'s. The connection matrix when restricting to $H_{f,\pm}$ with respect to the basis $\varphi_i$'s is written as $u^{-1} N$, where $N$ is consists of constant $1$-forms with coefficient in $V$. Treating $N$ as a matrix of $1$-forms in $\Omega_{\hat S^{\dag}}^1$, it gives the connection $1$-form with respect to the frame $DE(\varphi_i)$'s by our construction. We have the change of frames relation
	$$
	GA = u dG + N G
	$$ 
	relating $N$ and the connection matrix $A$ in Lemma \ref{lem:connection_special_form_semi_fano}. Now we can argue order by order that each $G_i$ is of the form 
	$$
	G_i = \begin{bmatrix}
	1 & 0 & 0 \\
	* & * & * \\
	* & * & * 
	\end{bmatrix}
	$$
	when written as a block matrix as in Lemma \ref{lem:connection_special_form_semi_fano}.  
	
	We can write $G_i$ as a matrix with coefficients in $R((u)) \setminus \mathcal{J}_i((u))$, and relate $G_{i+1} =G_i +  z^{p} \tilde{G}$ for $z^{p} \in \mathcal{J}_{i} \setminus \mathcal{J}_{i+1}$ and some $\tilde{G}$ with coefficients in $\C((u))$. Similarly we write $A_{i+1} = A_i + z^{p} \tilde{A}$ for the matrix $A$. Suppose the claim is true for $G_i$, and we consider the change of frames equation modulo $R((u))/\mathcal{J}_{i+1}$. We obtain
	$$
	(G_i + z^p \tilde{G})(A_i + z^{p} \tilde{A}) = u d(G_i + z^{p} \tilde{G}) + N(G_i + z^{p} \tilde{G}),
	$$
	and the relation 
	$$
	( ud\log(z^{p})\tilde{G}+ N\tilde{G}- \tilde{G}N)z^{p} =  G_i A_i - udG_i - NG_i + z^{p}\tilde{A}. 
	$$
	From the induction hypothesis and Lemma \ref{lem:connection_special_form_semi_fano}, we see that the RHS is a matrix with zero first row. Contracting with a constant vector field $n \in V^*$ such that $\iota d\log(z^{p}) \neq 0$, $\iota_n N$ becomes a lower triangular nilpotent matrix. Then $\tilde{G}$ is forced to have zero first row from the above equation. 
\end{proof}

\subsubsection{Explicit computations for the Hirzebruch surface $\mathbb{F}_2$}\label{F2}
    
For the Hirzebruch surface $\mathbb{F}_2$ (which is semi-Fano), perturbative expansions for the primitive form $\zeta$ and the semi-infinite period map can be calculated explicitly by hand.

Associated to $\mathbb{F}_2$ is the fan $\Sigma$ in $N_{\R} = \R^2$ whose $1$-dimensional cones are generated by $e_1 = (1,0)$, $e_2=(0,1)$, $e_3 = (-1,2)$ and $e_4 = (0,-1)$. Let $D_i \subset \mathbb{F}_2$ be the toric divisor corresponding to the ray spanned by $e_i$. The monoid $P$ is isomorphic to $\N^2$ with generators $C_1$ and $C_2$ corresponding to rational curves lying in $D_1$ and $D_2$ respectively. Here $C_1$ has Chern number $2$ and $C_2$ has Chern number $0$. The universal piecewise linear function $\varphi: |\Sigma| \rightarrow \R^2$ is defined by 
\begin{align*}
	\varphi(e_1) = (0,0) = \varphi(e_2),\quad \varphi(e_3) = (0,1), \quad \varphi(e_4) = (1,0). 
\end{align*}
Letting $q_1 = z^{(1,0)}$ and $q_2 = z^{(0,1)}$, we see that $q_1$ has weight degree $2$ and $q_2$ has weight degree $0$. $X^\dag$ is given as a subvariety of $\Spec (\C[q_1,q_2][z_1,z_2,z_3,z_4])$ by the relations
$$
z_2z_4 = q_1, \quad z_1z_3 = q_2 z_2^2. 
$$

To compare with known results, we will only compute the canonical coordinates and the primitive form when restricted to $\hat R = \C[[q_1,q_2]]$ (the small quantum variables), ignoring the other two parameters $t_1,t_2$. The Landau-Ginzburg superpotential is given by $F = z_1 + z_2 +z_3 +z_4$. The sheaf $\Omega_{\hat X^{\dag}/\hat S^{\dag}}^*$ of log differential forms is locally free with generators $d\log z_1$, $d\log z_2$, and the holomorphic volume form is taken to be $\Omega = d\log z_1 \wedge d\log z_2$. The cohomology $H^*(\Theta^*_{\hat X^{\dag}/\hat S^{\dag}}[[u]],u\partial + \{F,\cdot\})$ is concentrated at degree $0$ which is generated by polynomials in $z_i$'s, $q_i$'s and $u$.

Consider the action of $u\partial + \{F,\cdot\}$, it leads to the following identities
\begin{align*}
	(u\partial + \{F,\cdot\})(\tilde{\theta}_1) &= z_1-z_3 = 0, \\
	(u\partial + \{F,\cdot\})(\tilde{\theta}_2) &= z_2+2z_3-z_4=0, \\
	(u\pa + \{F,\cdot \})(z_2 (\tilde{\theta}_2+2\tilde{\theta}_1)) &= z_2^2 -q_1 +u z_2 +2z_1z_2 =0, \\
	 (u\pa + \{F,\cdot \})(z_1 \tilde{\theta}_2) &= z_1z_2+2q_2z_2^2 -z_1z_4 =0.
\end{align*}
They imply the following useful relations:
\begin{align*}
z_1 = z_3, \quad
z_2 = -2z_3 + z_4, \quad
z_2^2 = f(q_2)(q_1 - uz_2 - 2z_1z_4),
\end{align*}
where $f(q_2) = \frac{1}{1-4q_2}$.
We arrange the maximal cones $\sigma_i$'s in $\Sigma$ in the counter-clockwise ordering so that $\sigma_1 = \R_{\geq 0} e_1 + \R_{\geq 0} e_2$. This gives us frames $ 1$, $ z_3$, $ z_4$ and $z_1z_4$ in $H_{F,+}$ which restrict to $\varphi_i$'s in $H_{f,+}$ after modulo $q_1,q_2$, according to the choice made after equation \eqref{eqn:moving_lemma_condition}. 

Now we compute the Gauss-Manin connection acting on the frames $\psi_i$'s according to Section \ref{subsec:GM_connection_and_higher_residue}. We have 
\begin{align*}
\nabla_{\frac{\pa}{\pa \log q_2}} 1 & = \frac{1}{u} z_3,\\
\nabla_{\frac{\pa}{\pa \log q_2}} z_3 & = \frac{q_2 f(q_2)}{u} (q_1 - u(z_4-2z_3) -2z_1z_4),\\
\nabla_{\frac{\pa}{\pa \log q_2}} z_4 & = \frac{1}{u} z_1z_4,\\
\nabla_{\frac{\pa}{\pa \log q_2}} z_1z_4 & = \frac{q_1q_2}{u} (z_4-2z_3).
\end{align*}
These formula are obtained from equation \eqref{eqn:GM_connection_formula}. For instance, if we compute using the identification $\mathcal{H}_{F,\pm} \cong \mathcal{H}^{F}_\pm$ via $\Omega = d\log z_1 \wedge d\log z_2 = -d\log z_3 \wedge d\log z_2$, we have
\begin{align*}
\nabla_{\frac{\pa}{\pa \log q_2}} (z_3 \Omega) 
& = \iota_{\frac{\pa}{\pa \log q_2}}(\pa + u^{-1} dF\wedge)(z_3 d\log z_2 \wedge d\log z_3)\\
& =\iota_{\frac{\pa}{\pa \log q_2}} \frac{1}{u} z_1z_3 d\log z_1 \wedge d\log z_2 \wedge d\log z_3 \\
& = \iota_{\frac{\pa}{\pa \log q_2}} \frac{1}{u} q_2 z_2^2 d\log q_2 \wedge \Omega \\
& = \frac{q_2 f(q_2)}{u} (q_1 - u(z_4-2z_3) -2z_1z_4) \Omega.\\
\end{align*}
Similarly we have 
\begin{align*}
\nabla_{\frac{\pa}{\pa \log q_1}} 1 & = \frac{1}{u} z_4,\\
\nabla_{\frac{\pa}{\pa \log q_1}} z_3 & = \frac{1}{u} z_1z_4,\\
\nabla_{\frac{\pa}{\pa \log q_1}} z_4 & = \frac{1}{u}(q_1 + 2z_1z_4),\\
\nabla_{\frac{\pa}{\pa \log q_1}} z_1z_4 & = \frac{q_1(1+q_2)}{u}( 2z_4 -3z_3).
\end{align*}

Next we solve for the sections $DE(\varphi_i)$'s. We first restrict ourselves to $q_1 = 0$ and then extend to sections in $q_1$. Let $\psi_i$'s be the restrictions of $DE(\varphi_i)$'s to $q_1 = 0$. Then we can simply take $\psi_3 = z_4$ and $\psi_4 = z_1z_4$. It remains to solve for $\psi_1$ and $\psi_2$. To do so, we let $\psi_2 = a(q_2) z_3 + b(q_2)z_4 + c(q_2) z_1z_4$. Then the equation $\nabla_{\frac{\pa}{\pa \log q_2}} \psi_2 = 0$ gives us
\begin{align*}
a(q_2) & = (f(q_2))^{-1/2},\\
b(q_2) & = \frac{1}{2}(1-a(q_2)),\\
c'(q_2) & = \frac{1}{u} \left(2 (f(q_2))^{1/2} -\frac{b(q_2)}{q_2}\right), 
\end{align*}
where $c(q_2)$ is determined by the initial condition $c(0) =0$. Letting $\psi_1 = 1 + \alpha(q_2) z_3 + \beta(q_2) z_4 + \gamma(q_2) z_1z_4$, we have the equation $\nabla_{\frac{\pa}{\pa \log q_2}} \psi_1 = \frac{1}{u} \psi_2$ which gives us
\begin{align*}
\alpha'(q_2) & = \frac{1}{uq_2}\left(1-\frac{1}{a(q_2)}\right),\\
\beta'(q_2) &  = \frac{1}{uq_2}\frac{b(q_2)}{a(q_2)} = \frac{-1}{2}\alpha'(q_2),\\
\gamma'(q_2) & = \frac{1}{uq_2} \frac{c(q_2)}{a(q_2)},
\end{align*}
that is determined by the initial condition $\alpha(0) = \beta(0) = \gamma(0) = 0$. It is worth noting that $a(q_2), b(q_2) \in \C[q_2]$, and $\alpha(q_2) , \beta(q_2), c(q_2) \in \C[q_2] u^{-1}$ while $\gamma(q_2) \in  \C[q_2] u^{-2}$. 

To determine the primitive form, we express 
$$
1 = \sum_{i=1}^4 \delta_i \cdot DE(\varphi_i) 
$$
with $\delta_i = \delta_i(q_1,q_2,u) = \sum_{j=-\infty}^{\infty}\delta_{ij}(q_1,q_2)u^{j}$ such that  $\delta_{ij}(q_1,q_2) \in \C[[q_1,q_2]]$ and the summation starts from $j = -N_k$ if we modulo $(q_1,q_2)^{k+1}$. The weight degrees of $DE(\varphi_1)$, $DE(\varphi_2)$, $DE(\varphi_3)$ and $DE(\varphi_4)$ are $0$, $2$, $2$ and $4$ respectively. From these, we know that $\delta_1$ has weight degree $0$, $\delta_2$ and $\delta_3$ both have weight degree $-2$ and $\delta_4$ has weight degree $-4$. Since the weight degree of $q_1$ is $4$ and that of $q_2$ is $0$, we see that $\delta_{i,j\geq 0} = 0$ for $i>0$, $\delta_{0,j>0} = 0$ and $\delta_{0,0}$ is independent of $q_1$. Modulo $q_1$ we find that $\delta_{0,0} \cong 1$ from the expression of $\psi_1$. All in all, we see that the primitive form is simply given by $\zeta = 1$. 

Finally we compute the canonical coordinates using $\delta_{i,-1}(q_1,q_2)$. The weight degree requirement forces $\delta_{0,-1} = 0 = \delta_{4,-1}$, and $\delta_{2,-1}$, $\delta_{3,-1}$ are functions which depend only on $q_2$. Therefore, we see that $\delta_{2,-1}(q_2) = -\alpha(q_2) u$ and $\delta_{3,-1} = -\beta(q_2) u$. To compute the period map, we first compute 
\begin{equation}\label{eqn:flat_volume_element_example}
\begin{split}
\mathbb{1} = & DE(\varphi_1) - \log(q_2)\frac{DE(\varphi_2)}{u} - \log(q_1) \frac{DE(\varphi_3)}{u}  \\
& \qquad\qquad\qquad + (\log(q_1))^2 \frac{DE(\varphi_4)}{u^2} + \log(q_1) \log(q_2) \frac{DE(\varphi_4)}{u^2}. 
\end{split}
\end{equation}
Then we find the constant term in $u$ in the expression $u(\zeta-\mathbb{1})$, which will be
$$
(\log(q_2)-\alpha(q_2))DE(\varphi_2) +(\log(q_1)-\beta(q_2))DE(\varphi_3).
$$
The period map is then given by
\begin{align*}
\log (\tau_2) &= \log(q_2) - u \alpha(q_2) = \log(q_1) + 2 u \beta(q_2),\\
\log(\tau_3) & = \log(q_1) - u\beta(q_2),
\end{align*}
where we let $\log(\tau_3)$, $\log(\tau_2)$ be the coefficients of $DE(\varphi_3)$ and $DE(\varphi_2)$ respectively.
Taking $\frac{\pa}{\pa \log(q_2)}(\log \tau_2) = 1 - u\alpha'(q_2) = \frac{1}{a(q_2)} = f(q_2)^{1/2} $, and comparing with the known formulae
\begin{align*}
 q_2 &= \tau_2 (1+\tau_2)^{-2}, \\
 q_1 & = \tau_3(1+q_2)
\end{align*}
for the mirror map (in e.g., \cite{Chan-Lau10}), we conclude that the two expressions for $\tau_2(q_2)$ agree.

\begin{appendix}

\section{Dolbeault resolution}\label{Dolbeault}

This section is aimed at constructing Dolbeault resolutions of various sheaves on the analytic space $X$. We begin by recalling the definition of a polytopal complex, which is a generalization of the more familiar concept of a simplicial complex.
\begin{definition}
    A \emph{polytopal complex} $\Delta$ is a set of polytopes that satisfies the following conditions:
    \begin{enumerate}
        \item every face of a polytope in $\Delta$ is also in $\Delta$;
        \item the (possibly empty) intersection of any two polytopes $\sigma_1, \sigma_2 \in \Delta$ is a face of both $\sigma_1$ and $\sigma_2$.
     \end{enumerate}
\end{definition}
A polytopal complex is said to be \emph{finite} if it contains only finitely many polytopes. A polytope $\Delta$ together with its faces naturally defines a polytopal complex, which will still be denoted by $\Delta$.

\begin{definition}
    Let $\Delta$ be a finite polytopal complex. A \emph{(covariant) presheaf $\mathcal{V}$ of $\C$-vector space} on $\Delta$ consists of the following data:
    \begin{enumerate}
        \item for every polytope $\sigma$ of $\Delta$, a $\C$-vector space $\mathcal{V}(\sigma)$, and
        \item for every inclusion $\sigma_2 \subset \sigma_1$ of polytopes, a morphism of $\C$-vector spaces $r= r_{\sigma_1 \sigma_2}: \mathcal{V}(\sigma_2) \to \mathcal{V}(\sigma_1)$,
    \end{enumerate}
    such that the following conditions are satisfied:
    \begin{enumerate}
        \item $\mathcal{V}(\emptyset) = 0$,
        \item $r_{\sigma \sigma}$ is identity map on $\mathcal{V}(\sigma)$,
        \item if $\sigma_3 \subset \sigma_2 \subset \sigma_1$ are three polytopes of $\Delta$, then $r_{\sigma_1 \sigma_3} = r_{\sigma_1 \sigma_2} \circ r_{\sigma_2 \sigma_3}$.
    \end{enumerate}
\end{definition}

There is a \v{C}ech complex associated to the pair $(\Delta, \mathcal{V})$.
Equip each polytope in $\Delta$ with an orientation. Let $\Delta_p$ be the set of $p$-dimensional polytopes of $\Delta$.
For each $p\geq 0$, define $\mathcal{C}^p = \mathcal{C}^p(\Delta, \mathcal{V}) = \bigoplus_{\sigma \in \Delta_p} \mathcal{V}(\sigma)$, where the direct sum is taken over all polytopes in $\Delta_p$.
For $\alpha = (\alpha(\sigma)_{\sigma \in \Delta_p}) \in \mathcal{C}^p$, the \emph{combinatorial Čech differential} $\delta^p: \mathcal{C}^p \to \mathcal{C}^{p+1}$ is given by
\begin{equation*}
    (\delta^p \alpha)(\tau) = \sum_{\sigma \subset \tau} \pm r_{\tau \sigma}\alpha(\sigma)
\end{equation*}
where the sum if taken over all $p$ dimensional polytopes that is contained in $\tau$.
The sign here depends on the orientations of $\sigma$ and $\tau$: There are two orientations on $\sigma$, one is induced from that of  $\tau$, the other is the equipped orientation. If these two orientations on $\sigma$ agree, we take the positive sign, otherwise we take the negative sign.
Clearly $\delta^{p+1} \circ \delta^p = 0$, thus we get a Čech complex $(\mathcal{C}(\Delta, \mathcal{V}),\delta)$.

Take $\Delta$ to be the defining polytope of the projective toric manifold $Y$. Then each $p$-dimensional polytope $\sigma$ of $\Delta$ corresponds to an $(n-p)$-dimensional component of $X$, which is denoted by $\sigma^\circ$.
Moreover, the inverse inclusion relation holds, i.e., $\sigma \subset \tau$ if and only if $\tau^\circ \subset \sigma^\circ$.
For any $\sigma \in \Delta$, define
\begin{align*}
    \mathcal{V}(\sigma) := \mathscr{O}(\sigma^\circ)
\end{align*}
as the space of holomorphic functions on $\sigma^\circ$ and let $r$ be the restriction map induced by inclusion. Denote the resulting complex by $(\mathcal{C}(\Delta, \mathscr{O}),\delta)$.
Let $X_1,\cdots,X_k$ be $n$-dimensional components of $X$. Then $\mathcal{C}^0(\Delta, \mathscr{O}) = \bigoplus_{i=1}^k  \mathscr{O}(X_i)$. Let
\begin{align*}
    \mathscr{O}(X) \to \mathcal{C}^0(\Delta, \mathscr{O})
\end{align*}
be given by restriction map. Then we have the following complex:
\begin{equation} \label{check complex}
        0 \to \mathscr{O}(X) \to \mathcal{C}^0(\Delta, \mathscr{O}) \stackrel{\delta^0}{\longrightarrow} \mathcal{C}^1(\Delta, \mathscr{O}) \stackrel{\delta^1}{\longrightarrow} \mathcal{C}^2(\Delta, \mathscr{O}) \stackrel{\delta^2}{\longrightarrow} \cdots.
\end{equation}

\begin{proposition}\label{Check resolution_1}
    The complex \eqref{check complex} is exact.
\end{proposition}
\begin{remark}
    As each $p$-dimensional component $\sigma^\circ$ of $X$ is an intersection of coordinate hyperplanes, the canonical projection $\C^d \to \sigma^\circ$ lifts any function $g$ on $\sigma^\circ$ to a function on $\C^d$. Then by restriction we get a function on $X$ and each higher dimensional component containing $\sigma^\circ$. By abuse of notation, these functions will still be denoted by $g$.
\end{remark}
The proof of Proposition \ref{Check resolution_1} consists of three lemmas.
\begin{lemma} \label{p=0}
    The $p=0$ joint of the complex (\ref{check complex}) is exact.
\end{lemma}
\begin{proof}
     The injection part is obvious. To prove the exactness, we should show that every $\delta^0$-closed $(g_i) \in \mathcal{C}^0(\Delta, \mathscr{O})$ with $ g_i \in \mathscr{O}_{X_i}$ is the restriction of a global function on $X$.
     By definition, $\delta^0 (g_i) = 0$ implies $g_i|_{X_i \cap X_j} = g_j|_{X_i \cap X_j}$ if the intersection is $(n-1)$-dimensional. In fact, the same identity holds even if the intersection has codimension greater than $1$. This can be seen as follows. Since $X_i\cap X_j$ is a subset of both $X_i$ and $X_j$, both of the vertices $v_i$ and $v_j$ corresponding to $X_i$  and $X_j$ respectively are in the polytope $\sigma$ corresponding to $X_i\cap X_j$. We can find a sequence of vertices $v_i,v_s,v_t\cdots,v_u,v_j$ of $\sigma$ such that each pair of adjacent vertices is connected by an edge in $\sigma$. Now by the definition of $\delta^0$-closedness, we have
     \begin{align*}
         g_i|_{X_i\cap X_s} = g_s|_{X_i\cap X_s}, g_s|_{X_s\cap X_t} = g_t|_{X_s\cap X_t}, \cdots, g_u|_{X_u\cap X_j} = g_j|_{X_u\cap X_j}.
     \end{align*}
     However, we know that $X_i\cap X_j$ is a common subset of $X_i\cap X_s, X_s\cap X_t, \cdots$ and $X_u\cap X_j$, hence
     \begin{align*}
         g_i|_{X_i\cap X_j} = g_s|_{X_i\cap X_j} = \cdots = g_j|_{X_i\cap X_j}.
     \end{align*}
     Thus $(g_i)$ determines a unique function on each $\sigma^\circ$, which we denote by $(g_i)|_{\sigma^\circ}$.
     Define
     \begin{align*}
         g = \sum_{\sigma \in \Delta} (-1)^{\text{dim} \, \sigma} (g_i)|_{\sigma^\circ}.
     \end{align*}
     As remarked above, $g$ is seen as a global function on $X$, we will show that it restricts to $g_k$ on each $X_k$.
     Fix the vertex $v_k$, define $S_{v_k} := \{\tau : v_k \in \tau\}$.
     For $\tau \in S_{v_k}$, define $\tau^\bot := \{\sigma \in \Delta : \sigma \subset \tau, \sigma \nsubseteq \rho, \forall \rho \subset \tau, \rho \in S_{v_k}\}$.
     One sees easily by induction on dimension that the Euler characteristic $\chi(\tau^\bot)$ of each $\tau^\bot$ is zero unless $\tau = v_k$ and
     \begin{align*}
         \Delta = \coprod_{\tau \in S_{v_k}} \tau^{\bot}.
     \end{align*}
     Now $g$ can be written as
     \begin{align*}
         g = \sum_{\tau \in S_{v_k}} \sum_{\sigma \in \tau^{\bot}} (-1)^{\text{dim} \, \sigma} (g_i)|_{\sigma^\circ}.
     \end{align*}
     When $\tau \neq v_k$, for each $\sigma \in \tau^\bot$ the function induced on $X_k$ by $(g_i)|_{\sigma^\circ}$ is equal to $(g_i)|_{\tau^\circ}$. Hence the sum $\sum_{\sigma \in \tau^{\bot}} (-1)^{\text{dim} \, \sigma} (g_i)|_{\sigma^\circ}$ is just $\chi(\tau^\bot) (g_i)|_{\tau^\circ}=0$. It follows that $g = g_k$ on $X_k$.
\end{proof}

\begin{lemma} \label{parametrized solution}
    Assume matrices $A \in \C^{p\times q}, B\in \C^{q\times r}$ satisfy $\{\vec{a} \in \C^{q \times 1}: A \vec{a} = 0\} \subset \{B\vec{b} : \vec{b} \in \C^{r\times 1}\}$.
    Let $\vec{a}(z) := (a_1(z),\cdots,a_q(z))^T$ be a family of vectors that depends holomorphically on the parameters $z = (z_1,\cdots,z_k)$ and that for any $z$, $A \vec{a}(z)=0$.
    Then we can find vectors $\vec{b}(z)$ depending holomorphically on $z$ such that
    \begin{align*}
        B \vec{b}(z) = \vec{a}(z).
    \end{align*}
    Moreover, we can require $\vec{b}(z) = 0$ if $\vec{a}(z) = 0$.
\end{lemma}
\begin{proof}
    This is an easy exercise in linear algebra.
\end{proof}

\begin{lemma} \label{p>0}
    The $p>0$ joints of the complex (\ref{check complex}) are exact.
\end{lemma}
\begin{proof}
    Let $(g_{\sigma})_{\sigma \in \Delta_p}\in \mathcal{C}^p(\Delta, \mathscr{O})$ be $\delta^p$-closed with $p>0$, which means a holomorphic function $g_{\sigma}$ is assigned for each $p$-dimensional polytope and these functions satisfy certain compatibility conditions.
    Note that $\Delta$ is the unique $n$-dimensional polytope and it corresponds to the origin $o (= \Delta^\circ)$ of $\C^d$, and each function $g_{\sigma},\sigma \in \Delta_p$, restricts to a complex number $g_{\sigma}(0)$ at $o$.
    Look at the Čech complex $(\mathcal{C}(\Delta,\mathcal{V}), \delta)$ associated to the presheaf $\mathcal{V} := \mathscr{O}(o) = \C$.
    It's easy to see this complex computes the polyhedral cohomology of $\Delta$ and hence for $p>0$, the cohomologies vanish.
    Now $\delta^p$-closedness of $(g_{\sigma})_{\sigma \in \Delta_p}$ implies the $\delta^p$-closedness of $(g_{\sigma}(0))_{\sigma \in \Delta_p}$, hence we have
    \begin{align*}
        (g_{\sigma}(0))_{\sigma \in \Delta_p} = \delta^{p-1} (f^{\Delta}_{\tau})_{\tau \in \Delta_{p-1}}
    \end{align*}
    for some $(f^{\Delta}_{\tau})_{\tau \in \Delta_{p-1}} \in \mathcal{C}^{p-1}(\Delta, \mathscr{O}(o))$.
    The superscript $\Delta$ indicates the we are considering the restriction to $o = \Delta^\circ$.
    The projections $\tau^\circ \to o, \tau \in \Delta_{p-1}$ lift $(f^{\Delta}_{\tau})_{\tau \in \Delta_{p-1}}$ to chains in $\mathcal{C}^{p-1}(\Delta, \mathscr{O})$.
    Look at the $p$-cycle $(g^{(1)}_\sigma) := ((g_{\sigma}) - \delta^{p-1}(f^{\Delta}_\tau))$, we see by construction that the restriction of each $g^{(1)}_\sigma$ to $o$ is zero.

    Let $\rho \in \Delta_{n-1}$ be an $(n-1)$-dimensional polytope and consider the complex $(\mathcal{C}(\rho,\tilde{\mathscr{O}}(\rho^\circ)), \delta)$, where $\tilde{\mathscr{O}}(\rho^\circ)$ denotes the space of functions on $\rho^\circ$ that restricts to $0$ at $o$.
    Now closeness of $(g^{(1)}_\sigma)$ implies the closeness of the restriction of $(g^{(1)}_\sigma)_{\sigma \subset \rho}$ to a $p$-chain in $(\mathcal{C}(\rho,\tilde{\mathscr{O}}(\rho^\circ)), \delta)$ and the latter can be viewed as a family of $p$-cycles in $(\mathcal{C}(\rho,\C), \delta)$ that are parametrized by points of $\rho^\circ$.
    By Lemma \ref{parametrized solution}, we can find a $\rho^\circ$-parametrized $(p-1)$-chain $(f^{\rho}_{\tau})_{\tau \in \rho_{p-1}}$ in $(\mathcal{C}(\rho,\C), \delta)$ such that
    \begin{align*}
        \delta^{p-1} (f^{\rho}_{\tau})_{\tau \in \rho_{p-1}} = (g^{(1)}_\sigma)_{\sigma \subset \rho}.
    \end{align*}
    Moreover, $(f^{\rho}_{\tau})_{\tau \in \rho_{p-1}}$ is a chain in $(\mathcal{C}(\rho,\tilde{\mathscr{O}}(\rho^\circ)),\delta)$.
    If we define $f^{\rho}_{\tau} = 0$ for $\tau \in \Delta_{p-1}\setminus \rho_{p-1}$, we get now a $(p-1)$-chain $(f^{\rho}_{\tau})_{\tau \in \Delta_{p-1}}$ in $(\mathcal{C}(\Delta,\tilde{\mathscr{O}}(\rho^\circ)),\delta)$
    and the projection $\tau^\circ \to \rho^\circ$ allows us to view it as a chain in $(\mathcal{C}(\Delta,\mathscr{O}),\delta)$.
    Define
    \begin{align*}
        (g^{(2)}_\sigma) := ((g^{(1)}_{\sigma}) - \sum_{\rho \in \Delta_{n-1}} \delta^{p-1}(f^{\rho}_\tau)_{\tau \in \Delta_{p-1}}).
    \end{align*}
    We claim that the restriction of $(g^{(2)}_\sigma)$ to any $\rho^\circ,\rho \in \Delta_{n-1}$ is zero.
    By construction, $((g^{(1)}_{\sigma}) - \delta^{p-1}(f^{\rho}_\tau))$ restricts to zero on $\rho^\circ$.
    Given $\rho' \in \Delta_{n-1}, \rho' \neq \rho$, if there exists $\tau \in \rho'_{p-1}\cap \rho_{p-1}$, $f^{\rho'}_{\tau}$ is the pull-back to $\tau^\circ$ of a function in $\tilde{\mathscr{O}}(\rho'^\circ)$.
    As $\rho' \neq \rho$, $\rho^\circ \cap \rho'^\circ$ is zero dimensional, hence by construction $f^{\rho'}_\tau$ restrict to $0$ on $\rho^\circ$.
    If no $(p-1)$-dimensional polytope is in $\rho'_{p-1}\cap \rho_{p-1}$, then by definition, $f^{\rho'}_\tau = 0, \forall \tau \subset \rho$.

    Inductively, we can define $(g^{(3)}_\sigma), (g^{(4)}_\sigma),\cdots, (g^{(n-p)}_\sigma) \in (\mathcal{C}(\Delta,\mathscr{O}),\delta)$ such that: a) each $(g^{(i)}_\sigma)$ differs from $(g^{(i-1)}_\sigma)$ by a $\delta^{p-1}$-boundary; b) any function $g^{(i)}_\sigma$ restrict to $0$ on $\rho^\circ$, where $\rho$ is any polytope that contains $\sigma$ and is of dimension not smaller than $n-i+1$.

    Now $(g^{(n-p)}_\sigma)_{\sigma \in \Delta_p}$ differs from $(g_\sigma)_{\sigma \in \Delta_p}$ by a $\delta^{p-1}$-boundary and the lemma is proved if $(g^{(n-p)}_\sigma)_{\sigma \in \Delta_p}$ can also be written as a $\delta^{p-1}$-boundary.
    Given any $g^{(n-p)}_\sigma$ on $\sigma^\circ$, let $\tau(\sigma) \subset \sigma$ be a $(p-1)$-dimensional polytope.
    The projection $p: \tau(\sigma)^\circ \to \sigma^\circ$ lifts $g^{(n-p)}_\sigma$ to a function on $\tau(\sigma)^\circ$, which is denoted by $g^{(n-p)}_{\tau(\sigma)}$. Since $g^{(n-p)}_\sigma$ restrict to $0$ on $\rho^\circ$ for any $\rho$ properly contains $\sigma$, $g^{(n-p)}_{\tau(\sigma)}$ will restrict to $0$ on any $p$ dimensional polytope $\sigma' \in \Delta_p \setminus \{\sigma\}$ that contains $\tau(\sigma)$.
    Define a $(p-1)$-chain $(g^{(n-p)}_{\sigma;\tau})_{\tau \in \Delta_{p-1}}$ by letting $g^{n-p}_{\sigma;\tau(\sigma)} = g^{n-p}_{\tau(\sigma)}$ and  $g^{(n-p)}_{\sigma;\tau'} = 0$ for any $\tau' \in \Delta_{p-1} \setminus \{\tau(\sigma)\}$, then the $\sigma$-component of $\delta^{p-1} (g^{(n-p)}_{\sigma;\tau})$ is $\pm g^{(n-p)}_\sigma$, in which the sign depends on the orientations of $\sigma$ and $\tau$.
    We put the sign before $(g^{(n-p)}_{\sigma;\tau})$, which by abuse of notation we still denote by $(g^{(n-p)}_{\sigma;\tau})$, so that the $\sigma$-component of $\delta^{p-1} (g^{(n-p)}_{\sigma;\tau})_{\tau \in \Delta_{p-1}}$ is exactly $g^{(n-p)}_\sigma$.
    Taking summation over all $\sigma \in \Delta_{p-1}$, we finally have
    \begin{align*}
        (g^{(n-p)}_\sigma)_{\sigma \in \Delta_p} = \delta^{p-1} \sum_{\sigma \in \Delta_p} (g^{(n-p)}_{\sigma;\tau})_{\tau \in \Delta_{p-1}},
    \end{align*}
    i.e., $(g^{(n-p)}_\sigma)_{\sigma \in \Delta_p}$ is a $\delta^{p-1}$-boundary. Thus the lemma is proved.
\end{proof}

Combining Lemmas \ref{p=0} and \ref{p>0} gives Proposition \ref{Check resolution_1}.

For any $ 0 < p \leq n$, we can similarly define the complex $(\mathcal{C}(\Delta, \Omega^p), \delta)$, where the presheaf $\mathcal{V} := \Omega^p$ is defined by assigning to $\sigma$ the space $\Omega(\sigma^\circ)$ of holomorphic $p$-forms on $\sigma^\circ$ and the map $r$ is the restriction map.
Let $\tilde{\Omega}^p(X)$ be the space of holomorphic $p$-forms on $X$, or in other words,
\begin{align*}
    \tilde{\Omega}^p(X) := \{(w_1,\cdots,w_k) : w_i \in \Omega^p(X_i), \omega^i|_{X_i \cap X_j} = \omega^j|_{X_i \cap X_j}, \forall \text{dim}(X_i \cap X_j) = n-1\}
\end{align*}
and let $\tilde{\Omega}^p(X) \to \mathcal{C}^0(\Delta, \Omega^p)$ be the natural map given by restriction. We obtain the following complex:
\begin{equation} \label{check complex_2}
        0 \to \tilde{\Omega}^p(X) \to \mathcal{C}^0(\Delta, \Omega^p) \stackrel{\delta^0}{\longrightarrow} \mathcal{C}^1(\Delta, \Omega^p) \stackrel{\delta^1}{\longrightarrow} \mathcal{C}^2(\Delta, \Omega^p) \stackrel{\delta^2}{\longrightarrow} \cdots.
\end{equation}

\begin{proposition} \label{Check resolution_2}
    The complex (\ref{check complex_2}) is exact.
\end{proposition}
\begin{proof}
    In fact, only two types of maps appeared in the proof of Proposition \ref{Check resolution_1}: the pull-back map induced from projection and the restriction map induced from inclusion.
    The key point is that for $\sigma, \tau \in \Delta$, a function on $\sigma^\circ$ induces the same function on $\tau^\circ$ ether by first pulling back to $\C^d$ then restricting to $\tau^\circ$ or first restricting to $\sigma^\circ \cap \tau^\circ$ and then pulling back to $\tau^\circ$.
    The situation is the same for differential forms, hence the proof runs in the same way.
\end{proof}

In the definitions of the complexes \eqref{check complex} and \eqref{check complex_2}, holomorphic functions can be replaced by smooth functions and holomorphic $p$-forms can be replaced by anti-holomorphic $q$-forms. Let $\A^{0,q}(X)$ be the space of smooth $(0,q)$-forms on $X$. Defining the complex $(\mathcal{C}(\Delta, \A^{0,q}), \delta)$ and the corresponding maps in a similar way as before, we can prove the following theorem along the same line of the proofs of Propositions \ref{Check resolution_1} and \ref{Check resolution_2}.

\begin{proposition} \label{Check resolution_3}
    For any $0 \leq q \leq n$, we have the following exact complex
    \begin{equation} \label{check complex_3}
        0 \to \A^{0,q}(X) \to \mathcal{C}^0(\Delta, \A^{0,q}) \stackrel{\delta^0}{\longrightarrow} \mathcal{C}^1(\Delta, \A^{0,q}) \stackrel{\delta^1}{\longrightarrow} \mathcal{C}^2(\Delta, \A^{0,q}) \stackrel{\delta^2}{\longrightarrow} \cdots.
    \end{equation}
\end{proposition}

\begin{theorem} [Dolbeaut resolution of $\mathscr{O}(X)$] \label{Dolbeault_resolution}
    The complex
    \begin{equation} \label{Dolbeault resolution_1}
        0 \to \A^{0,0}(X) \stackrel{\bar\pa}{\longrightarrow} \A^{0,1}(X) \stackrel{\bar\pa}{\longrightarrow} \A^{0,2}(X) \stackrel{\bar\pa}{\longrightarrow} \cdots
    \end{equation}
    is a resolution of $\mathscr{O}(X)$.
\end{theorem}

\begin{proof}
    Since $\bar\pa$ commutes with the restriction map, we can construct a double complex $(C^{*,*}, \td_1,\td_2)$ with
    \begin{align*}
        C^{p,q} := \mathcal{C}^p(\Delta, \A^{0,q}), \td_1 := \delta^p, \td_2 := (-1)^q\bar\pa.
    \end{align*}
    Using the fact that each $\sigma^\circ$ is affine, one can first compute the cohomologies of $\td_2$ to show that the only possible nonzero terms on the first page of the spectral sequence associated to the filtration given by index $p$ is $^{I}E_1^{p,0} = \mathcal{C}^p(\Delta, \mathscr{O})$.
    So the spectral sequence degenerate at the second page and the only nonzero term is $^{I}E_2^{0,0} = \mathscr{O}(X)$ by applying Proposition \ref{Check resolution_1}.
    One can also first compute the cohomologies of $\td_1$ by applying Proposition \ref{Check resolution_3}. The result is that the only possible nonzero terms on the first page of the spectral sequence associated to the filtration given by index $q$ is $^{II}E_1^{0,q} = \tilde{\A}^{0,q}(X)$. So the spectral sequence must also degenerate at the second page and one must have $^{II}E_2^{0,0} = \mathscr{O}(X)$.
\end{proof}

\section{Regularized integral} \label{Section of regularized integral}

In this appendix, we outline a generalization of the theory of regularized integrals due to Si Li and Jie Zhou \cite{LZ} to arbitrary dimensions. Since no additional essential difficulty appears, we will omit the proofs and cite the corresponding results in \cite{LZ} instead.

Let $M$ be a compact complex manifold of dimension $n$, possibly with boundary $\pa M$. Let $D$ be a simple normal crossing divisor in $M$ which does not meet $\pa M$.
Let
\begin{equation*}
    \Omega_M^\bullet (\star D) := \bigcup_{n\geq 0} \Omega_M^\bullet (nD)
\end{equation*}
be the sheaf of meromorphic forms which are holomorphic on $M-D$ but possibly with arbitrary orders of poles along $D$.
Denote
\begin{equation*}
    \A^{p,q}(M,\star D) := \A^{0,q}(M, \Omega^p(\star D)), \, \A^k(M,\star D) := \bigoplus_{p+q = k} \A^{p,q}(M,\star D).
\end{equation*}
By definition, $\omega \in \A^k(M,\star D)$ if and only if $\omega$ is smooth on $M-D$ and locally around any $p \in D$, $\omega$ is of the form
\begin{equation*}
    \frac{\alpha}{z_1^{m_1} \cdots z_l^{m_l}};
\end{equation*}
here $z_1,\cdots,z_n$ are local coordinates around $p$ such that $D$ is defined by the equation $z_1 \cdots z_l = 0, l \leq n$, $m_i$'s are non-negative integers and $\alpha$ is a smooth $k$-form around $p$.

The complex $\A^{\bullet,\bullet}(M,\star D)$ is a bi-graded complex with natural differentials $\pa$ and $\bar \pa$. Moreover,
\begin{equation*}
    \A^{\bullet,\bullet}(M,\log D) \subset \A^{\bullet,\bullet}(M,\star D)
\end{equation*}
is a bi-graded subcomplex. The following results are the counterparts of Lemma 2.1 and Theorem 2.4 in \cite{LZ} respectively.

\begin{lemma}
Any $\omega \in \A^{n,\bullet}(M,\star D)$ can be written as
\begin{align*}
    \omega = \alpha + \pa \beta, \text{ where } \alpha \in \A^{p,\bullet}(M,\log D), \beta \in \A^{n-1,\bullet}(M,\star D).
\end{align*}
The support of $\alpha, \beta$ can be chosen to be contained in the support of $\omega$.
\end{lemma}

\begin{theorem} \label{Regularized integral}
Let $\omega \in \A^{2n}(M,\star D)$. Then there exist $\alpha \in \A^{2n}(M,\log D)$ and $\beta \in \A^{n-1,n}(M,\star D)$ such that $\omega = \alpha + \pa \beta$. The integral $\int_M \alpha$ is absolutely convergent and the sum
\begin{equation*}
    \int_M \alpha + \int_{\pa M} \beta
\end{equation*}
does not depend on the choices of $\alpha$ and $\beta$.
\end{theorem}

\begin{definition}\label{regularized integral}
We define the \emph{regularized integral}
\begin{equation*}
    \dashint_M\omega:=
    \begin{cases}
        0 & \text{if}\quad \omega \in \A^{<2n}(M, \star D)\,,\\
        \int_M\alpha+ \int_{\pa M} \beta & \text{if}\quad \omega=\alpha+\pa \beta\in \A^{2n}(M, \star D)\,.
    \end{cases}
\end{equation*}
Here $\alpha \in \A^{2n}(M,\log D), \beta \in \A^{n-1,n}(M,\star D)$.
\end{definition}

Parallel to \cite[Theorem 2.4]{LZ}, the proof of Theorem \ref{Regularized integral} relies on the vanishing of the limit of a contour integral.
A careful investigation of this contour integral leads to the following version of Poincar\'e residue map. Let $U$ be a small neighbourhood of $p \in D$ with coordinates $z_1,\cdots,z_n$ such that
\begin{equation*}
    D \cap U = \{z_1\cdots z_l = 0\}
\end{equation*}
for some $l \leq n$.
For any $\omega \in \A^{p,q}(M,\star D)$, we can write
\begin{equation}\label{eqn:higher_residue_definition_decomposition}
    \omega = \td \bar z_1 \wedge \alpha + \frac{\td z_1}{z_1^{m_1}} \wedge \beta + \gamma,
\end{equation}
where $\beta$ does not contain $\td \bar z_1$ and $\gamma$ does not contain $\td z_1$ or $\td \bar z_1$.
Furthermore, we can write
\begin{equation*}
    \beta = \sum_{J,K: 1 \notin J, 1 \notin K} h_{JK} \td z_J \wedge \td \bar z_K,
\end{equation*}
where $\td z_J = \wedge_{j\in J} dz_j, \td \bar z_K = \wedge_{k\in K} \td \bar z_k$.
Assume $D_1 \cap U = \{z_1 = 0\}$, then we can put
\begin{equation*}
    \Res_{D_1}(\omega) := \sum_{J,K: 1 \notin J, 1 \notin K} \frac{1}{(m_1-1)!} (\frac{\pa}{\pa z_1})^{m_1-1} h_{JK}|_{z_1=\bar z_1=0} \td z_J \wedge \td \bar z_K
\end{equation*}
on $D_1 \cap U$. It is straightforward to verify that $\Res_{D_1}(\omega)$ is globally defined and we thus get
\begin{equation*}
    \Res_{D_1}(\omega) \in \A^{p-1,q}(D_1, \star (D_1 \cap \cup_{j\neq 1} D_j)).
\end{equation*}
In the same way we can define $\Res_{D_i}(\omega), 2 \leq i \leq l$.
By definition, for any $\omega \in \A^{p,q}(M,\star D)$ we have
\begin{equation} \label{vanishing residue}
    \Res_{D_i}(\bar z_i \omega) = \Res_{D_i}(\td \bar z_i \wedge \omega) = 0.
\end{equation}

\begin{theorem} \label{Stokes formula 1}
Let $M$ be a compact complex manifold possibly with boundary $\pa M$. Let $\omega \in \A^{2n-1}(M, \star D)$. Then we have the following version of Stokes formula for the regularized integral
\begin{equation*}
    \dashint_M d \omega=-2\pi i \sum_{j=1}^l \dashint_{D_j} \Res_{D_j}(\omega)+\int_{\pa M}\omega\,.
\end{equation*}
\end{theorem}

If $M$ is a noncompact complex manifold without boundary, letting $\A_c^{\bullet,\bullet}(M,\star D) \subset \A^{\bullet,\bullet}(M,\star D)$ be the subspace of forms with compact support, we can then define the regularized integral of $\omega \in \A_c^{\bullet,\bullet}(M,\star D)$ on $M$.
In particular, we have the following Stokes formula.

\begin{theorem} \label{Stokes formula 2}
Let $M$ be a noncompact complex manifold and $\omega \in \A_c^{2n-1}(M, \star D)$. Then
\begin{equation*}
    \dashint_M d \omega=-2\pi i \sum_{j=1}^l \dashint_{D_j} \Res_{D_j}(\omega)\,.
\end{equation*}
\end{theorem}

\end{appendix}


\end{document}